\documentclass[a4paper,12pt]{article}
\usepackage{amsthm,mathrsfs}
\usepackage{amssymb,amsmath,latexsym,times}

\usepackage[pdftex]{hyperref}
\hypersetup{plainpages=True, pdfstartview=FitH,
colorlinks=true,linkcolor=blue,citecolor=blue,bookmarks=true}

\newtheorem{thm}{Theorem}[section]
\newtheorem{cor}[thm]{Corollary}
\newtheorem{lmm}[thm]{Lemma}
\newtheorem{prop}[thm]{Proposition}

\theoremstyle{definition}

\theoremstyle{remark}
\newtheorem{rem}[thm]{Remark}

\numberwithin{equation}{section}
\newcounter{cc}

\def\Po{{\mathscr{P}}}

\def\dtv{d_{T\!V}}
\def\dkl{d_{K\!L}}

\def\dd#1{{\rm d}#1}

\def\le{\leqslant}
\def\ge{\geqslant}

\def\dd#1{{\mathrm d}#1}

\def\ve{\varepsilon}

\title{A Charlier-Parseval approach to Poisson approximation\\ and its
applications}
\author{Vytas Zacharovas, Hsien-Kuei Hwang \\
        Institute of Statistical Science\\
        Academia Sinica\\ Taipei 115\\ Taiwan}
\date{\today}

\begin{document}
\maketitle


\begin{abstract}
A new approach to Poisson approximation is proposed. The basic idea
is very simple and based on properties of the Charlier polynomials
and the Parseval identity. Such an approach quickly leads to new
effective bounds for several Poisson approximation problems. A
selected survey on diverse Poisson approximation results is also
given.
\end{abstract}



\maketitle

\section{Introduction}

Poisson approximation to many discrete distributions (notably the
Poisson-binomial distribution) has received extensive attention in
the literature and many different approaches have been proposed. The
main problem is to study the closeness between the discrete
distribution in question and a suitably chosen Poisson distribution.
Applications in diverse problems also stimulated much of its recent
interest among probabilists and scientists in applied disciplines.
We propose in this paper a new, self-contained approach to Poisson
approximation, which leads readily to many new effective bounds for
several distances studied before, including total variation,
Kolmogorov, Wasserstein, Kullback-Leibler, point metric, and
$\chi^2$; see below for more information and references. In addition
to the application to these distances, we also attempt to survey
most of the quantitative results we collected for the Poisson
approximation distances discussed in this paper.

\subsection{A historical account with brief review of results}

We start with a brief historical account of Poisson approximation,
focusing particular on the evolution of the total variation
distance; a more detailed, technical discussion will be given in
Section~\ref{sec:cd}. For other surveys, see
\cite{Haight67,BHJ92,Barbour01,CDM05,Roos07}.

\paragraph{The early history of Poisson approximation.} Poisson
distribution appeared naturally as the limit of the sum of a large
number of independent trials each with very small probability of
success. Such a limit form, being the most primitive version of
Poisson approximation, dates back to at least de Moivre's work
\cite{de-Moivre56} in the early eighteenth century and Poisson's
book \cite{Poisson37} in the nineteenth century. Haight
\cite{Haight67} writes: ``\dots although Poisson (or de Moivre)
discovered the mathematical expression (1.1-1) [which is
$e^{-\lambda}\lambda^k/k!$], Bortkiewicz discovered the probability
distribution (1.1-1)." And according to Good \cite{Good86},
``perhaps the Poisson distribution should have been named after von
Bortkiewicz (1898) because he was the first to write extensively
about rare events whereas Poisson added little to what de Moivre had
said on the matter and was probably aware of de Moivre's work;" see
also Seneta's account in \cite{Seneta83} on Abbe's work. In addition
to Bortkiewicz's book \cite{Bortkiewicz98}, another important
contribution to the early history of Poisson approximation was made
by Charlier \cite{Charlier05} for his type B expansion, which will
play a crucial role in our development of arguments.

The next half a century or so after Bortkiewicz and Charlier then
witnessed an increase of interests in the properties and
applications of the Poisson distribution and Charlier's expansion.
In particular, Jordan \cite{Jordan26} proved the orthogonality of
the Charlier polynomials with respect to the Poisson measure, and
considered a formal expansion pair, expressing the Taylor
coefficients of a given function in terms of series of Charlier
polynomials and vice versa. A sufficient condition justifying the
validity of such an expansion pair was later on provided by Uspensky
\cite{Uspensky31}; he also derived very precise estimates for the
coefficients in the case of binomial distribution. His
complex-analytic approach was later on extended by Shorgin
\cite{Shorgin77} to the more general Poisson-binomial distribution
(each trial with a different probability; see next paragraph).
Schmidt \cite{Schmidt33} then gives a sufficient and necessary
condition for justifying the Charlier-Jordan expansion; see also
Boas \cite{Boas49} and the references therein. Prohorov
\cite{Prohorov53} was the first to study, using elementary
arguments, the total variation distance between binomial and Poisson
distributions, thus upgrading the classical limit theorem to an
approximation theorem.

\paragraph{From classical to modern.} However, a large portion of the
development of modern theory of Poisson approximation deviates
significantly from the classical line, and much of its modern
interest can be attributed to the pioneering paper by Le Cam
\cite{LeCam60}, extending the previous study by Prohorov
\cite{Prohorov53} for binomial distribution. Le Cam considered
particularly the sum $S_n$ of $n$ independent Bernoulli random
variables with parameters $p_1,p_2,\dots, p_n$, respectively, and
proved that the total variation distance
\[
    \dtv(\mathscr{L}(S_n),\Po(\lambda)) := \frac12
    \sum_{j\ge0} \left|\mathbb{P}(S_n=j) - e^{-\lambda}
    \frac{\lambda^j}{j!}\right|
\]
between the distribution of $S_n$ (often referred to as the
Poisson-binomial distribution) and that of a Poisson with mean
$\lambda := \sum_{1\le j\le n} p_j$ is bounded above by
\[
    \dtv(\mathscr{L}(S_n),\Po(\lambda)) \le 8\theta,
\]
whenever $p_* := \max_j p_j\le 1/4$, where $\theta :=
\lambda_2/\lambda$, $\lambda_2 := \sum_{1\le j\le n} p_j^2$. He also
proved in the same paper the following inequality, now often
referred to under his name,
\begin{align} \label{lecam-ineq}
    \dtv(\mathscr{L}(S_n),\Po(\lambda)) \le \lambda_2.
\end{align}
These results were later on further improved in the literature and
the approach he used became the source of developments of more
advanced tools; see Table~\ref{tb1} for a selected list of known
results of the simplest form $\dtv \le c \theta$.

\begin{table}[!h]
\begin{center}
\begin{tabular}{|c|c|c|c|c|}\hline
    Author(s) & Year & $\dtv\le $ & Assumption &
    Approach \\ \hline\hline
    Le Cam & 1960 & $ 8\theta$ & $p_*\le \frac14$
    & Operator and Fourier \\ \hline
    Kerstan & 1964 & $ 1.05\theta$ & $p_*\le \frac14$
    & Operator and Fourier \\ \hline
    Chen  & 1974 & $ 5\theta$ &  & Chen-Stein \\ \hline
    Barbour and Hall & 1984 & $ \theta$
    & & Chen-Stein\\ \hline
    Presman & 1985 & $ 2.08\theta$ & & Fourier\\ \hline
    Daley and Vere-Jones & 1988 & $ 0.71\theta$ &
    $p_*\le \frac14$ & Fourier \\ \hline
\end{tabular}
\end{center}
\label{tb1} \caption{\emph{Some results of the form
$\dtv:=\dtv(\mathscr{L}(S_n),\Po(\lambda)) \le c\theta$. Here
$\theta := \lambda_2/\lambda$ and $p_* := \max_j p_j$. It is known
that $\dtv(\mathscr{L}(S_n),\Po(\lambda)) \sim \theta/\sqrt{2\pi e}$
when $\theta\to0$; see Deheuvels and Pfeifer \cite{DP88} or Hwang
\cite{Hwang99}. Numerically, $1/\sqrt{2\pi e} \approx 0.242$.}}
\end{table}

Form Table~\ref{tb1}, we should point out that the leading constant
in the first-order estimate for $\dtv$ is often less important than
the generality of the approach used, although the pursuit for
optimal leading constant is of independent interest \emph{per se}.
One reason is that if an approach is quickly amended for obtaining
higher-order estimates, then one can push the calculations further
by obtaining more terms in the asymptotic expansions with smaller
and smaller errors, so that the implied constants in the error terms
matter less (the derivation of which often involves detailed
calculus).

On the other hand, estimates for the total variation distance
between the distribution of $S_n$ and a suitably chosen Poisson
distribution has been the subject of many papers in the last five
decades. Other forms in the literature include $\dtv \le
\varphi(\theta)$, $\dtv\le \varphi(\theta,\max_j p_j)$, $\dtv \le
\varphi(\theta,\lambda)$, $\dots$, for certain functionals $\varphi$
($\varphi$ not the same for each occurrence). Thus it is often
difficult to compare these results; further complications arise
because some metrics are related to others by simple inequalities
and the results for one can be transferred to the others; also the
complexity of the diverse methods of proof is not easily compared.
Despite these, we quickly review those that are pertinent to ours, a
more detailed, technical comparative discussion for some of these
will be given later; the special case of binomial distribution will
however not be compared separately; see, for example, Prohorov
\cite{Prohorov53}, Vervaat \cite{Vervaat69}, Romanowska
\cite{Romanowska77}, Matsunawa \cite{Matsunawa82}, Pfeifer
\cite{Pfeifer85}, Kennedy and Quine \cite{KQ89}, Poor \cite{Poor91}.

Kerstan \cite{Kerstan64} refined some results of Le Cam
\cite{LeCam60} on $\dtv$ by a similar approach. He also derived a
second-order estimate. Herrmann \cite{Herrmann65} further extended
results in Kerstan \cite{Kerstan64} in two directions: to \emph{sums
of random variables each assuming finitely many integer values} and,
in addition to higher-order estimates from the Charlier expansion,
to \emph{signed measures} whose generating functions are of the
forms $\exp(\sum_{1\le j\le s} (-1)^{j-1}\lambda_j (z-1)^j/j)$. We
will comment on Kerstan's and Herrmann's second-order estimates
later. As far as we are aware, Herrmann \cite{Herrmann65} was the
first to use such signed measures for Poisson approximation
problems, although such approximations are later on referred to as
Kornya-Presman or Kornya-type approximations, the two references
being Kornya \cite{Kornya83} and Presman \cite{Presman83}. Note that
the idea of using other signed measures (binomial) were already
discussed in Le Cam \cite{LeCam60}. Serfling \cite{Serfling74}
extended Le Cam's inequality (\ref{lecam-ineq}) to dependent cases;
see also \cite{Serfling78}. Chen \cite{Chen75} proposed a new
approach to Poisson approximation, based on Stein's method of normal
approximation (see Stein \cite{Stein86}).

From 1980 on, most of the approaches proposed previously for Poisson
approximation problems received much more attention and were further
developed and refined. Among these, \emph{the Chen-Stein method}
(with or without \emph{couplings}) is undoubtedly the most widely
used and the most fruitful one. It is readily amended for dealing
with dependent situations, but leads usually to less precise bounds
for numerical purposes. On the other hand, direct or indirect
classical \emph{Fourier analysis}, although involving less
probability ingredient and relying on more explicit forms of
generating functions, often gives better numerical bounds. For these
and other approaches (including \emph{semigroup} with Fourier
analysis, \emph{information-theoretic}), see Deheuvels and Pfeifer
\cite{DP86a}, Stein \cite{Stein86}, Aldous (1989), Barbour et al.\
\cite{BHJ92}, Steele \cite{Steele94}, Janson \cite{Janson94}, Roos
\cite{Roos99,Roos01}, Kontoyiannis et al.\ \cite{KHJ05} and the
references therein.

\subsection{Our new approach}

The new approach we are developing in this paper starts from the
integral representation for a given sequence $\{A_n\}_{n\ge0}$
(satisfying certain conditions specified in the next section)
\begin{align} \label{CP0}
    \sum_{n\ge0}\left|\frac{A_n}{e^{-\lambda}
    \frac{\lambda^n}{n!}}\right|^2e^{-\lambda}
    \frac{\lambda^n}{n!} =\int_0^\infty
    e^{-r}I(\sqrt{r/\lambda})\,dr,
\end{align}
where $\lambda>0$ and
\[
    I(r):=\frac{1}{2\pi}\int_{-\pi}^\pi\left
    |e^{-\lambda re^{it}}\sum_{j\ge0}
    A_j(1+re^{it})^j\right|^2\,dt.
\]
Note that $I(r) = \sum_{n\ge0} |a_n|^2 r^{2n}$, where $a_n$ denotes
the coefficient of $z^n$ in the Taylor expansion of $e^{-\lambda
z}\sum_{j\ge0} A_j (1+z)^j$. This means that (\ref{CP0}) can be
written in the form
\[
    \sum_{n\ge0}\frac{|A_n|^2}{e^{-\lambda}
    \frac{\lambda^n}{n!}}= \sum_{n\ge0} |a_n|^2
    \frac{n!}{\lambda^n},
\]
which, as far as we are aware, already appeared in the paper
Pollaczek-Geiringer \cite{PG28}, but no further use of it has been
discussed; see also Jacob \cite{Jacob33}, Schmidt \cite{Schmidt33},
Siegmund-Schultze \cite{SS93} and the references cited there. Also
the series on the right-hand side is in almost all cases we are
considering less useful than the integral in \ref{CP0}.

The seemingly strange and complicated starting point (\ref{CP0})
turns out to be very useful for developing effective tools for most
Poisson approximation problems. Other ingredients required are
surprisingly simple, with very little use of complex analysis. A
typical result is of the form
\[
    \dtv (\mathscr{L}(S_n),\Po(\lambda)) \le
    \frac{(\sqrt{e}-1)\theta}{\sqrt{2}(1-\theta)^{3/2}},
\]
where $(\sqrt{e}-1)/\sqrt{2}\approx 0.46$; see Theorem~\ref{thm-P0}.
The relation (\ref{CP0}), which will be proved below, is based on
the orthogonality of Charlier polynomials and Parseval identity;
thus we call it \emph{the Charlier-Parseval identity}.

Other features of our approach are: first, it reduces the estimate
of the probability distances to that of certain integral
representations with a similar form to the right-hand side of
(\ref{CP0}), and thus being of certain Tauberian character; second,
it can be readily extended to derive asymptotic expansions; third,
the use of the correspondence between Charlier polynomials and
Poisson distribution can be quickly amended for other families of
orthogonal polynomials and their corresponding probability
distributions; fourth, the same idea used applies equally well to
the de-Poissonization procedure, and leads to some interesting new
results, details being discussed elsewhere.

\paragraph{Organization of the paper.} This paper is organized as
follows. We begin with the development of our approach in the next
section. Then except for Section~\ref{sec:cd}, which is focusing on
reviewing and comparing with known results, the next three sections
consist of applications of our Charlier-Parseval approach:
Section~\ref{sec:App1} to several distances of Poisson approximation
to $S_n$ for large $\lambda$, Section~\ref{sec:App2} to second order
estimates, Section~\ref{sec:App3} to approximations by signed
measures.

\section{The new Charlier-Parseval approach}

Crucial to the development our approach is the use of Charlier
polynomials, so we first derive a few properties of Charlier
polynomials we will need.

\subsection{Definition and basic properties of Charlier polynomials}
The Charlier polynomials $C_k(\lambda,n)$ are defined by
\begin{equation}\label{def-char}
    \sum_{n\ge0}C_k(\lambda,n) \frac{\lambda^n}{n!}\,z^n
    =(z-1)^ke^{\lambda z}\qquad(k=0,1,\dots).
\end{equation}
Multiplying both sides by $z-1$, we see that
\begin{equation}\label{diff-eq}
    \frac{\lambda^{n-1}}{(n-1)!}C_k(\lambda,n-1)
    -\frac{\lambda^n}{n!}C_k(\lambda,n)
    =\frac{\lambda^n}{n!}C_{k+1}(\lambda,n),
\end{equation}
which implies that the Charlier polynomials
$\varphi_k(n):=C_{k}(\lambda,n)$ are solutions to the system of
difference equations $x\varphi_k(x-1)-\lambda \varphi_k(x)=\lambda
\varphi_{k+1}(x)$, with the initial condition $\varphi_0(x)\equiv
1$. In particular,
\begin{align} \label{C1C2}
    C_1(\lambda,n) = \frac{n-\lambda}{\lambda}\quad \text{and}
    \quad C_2(\lambda,n) = \frac{n^2-(2\lambda+1)n+\lambda^2}
    {\lambda^2}.
\end{align}

An alternative expression for $C_k(\lambda,n)$ is given by
\[
    \frac{\lambda^n}{n!}C_k(\lambda,n)
    =e^\lambda\frac{d^k}{d\lambda^k}\,
    e^{-\lambda}\frac{\lambda^n}{n!},
\]
which follows from substituting the relation $(z-1)^ke^{\lambda z}
=e^\lambda (d^k/d\lambda^k)\,e^{\lambda (z-1)}$ into
(\ref{def-char}).

Since by (\ref{def-char})
\begin{align} \label{Cklm}
    C_k(\lambda,n) \frac{\lambda^n}{n!}
    =[z^n](z-1)^ke^{\lambda z},
\end{align}
where $[z^n]\phi(z)$ denotes the coefficient of $z^n$ in the Taylor
expansion of $\phi(z)$, we have, for each fixed $n$,
\[
\begin{split}
    \frac{\lambda^n}{n!}\sum_{k\ge0}
    \frac{\lambda^k}{k!}C_k(\lambda,n)w^k
    &=[z^n] \sum_{k\ge0}
    \frac{\lambda^k}{k!}w^k(z-1)^ke^{\lambda z}\\
    &=[z^n]e^{-\lambda w+z\lambda(w+1)}\\
    &=\frac{\lambda^n}{n!}(1+w)^me^{-\lambda w}.
\end{split}
\]
It follows that
\[
    \sum_{n\ge0}C_n(\lambda,k) \frac{\lambda^n}{n!}\,w^n
    = (1+w)^ke^{-\lambda w}.
\]
Comparing this relation with (\ref{def-char}), we obtain the
property $C_k(\lambda,n)=(-1)^{n+k}C_n(\lambda,k)$, for all $k,n\ge
0$.

Another important property we will need is the following
orthogonality relation (see \cite[p.\ 35]{Szego39}).
\begin{lmm} The Charlier polynomials are orthogonal with respect to
the Poisson measure $e^{-\lambda}{\lambda^n}/{n!}$, namely,
\begin{equation}\label{ortho}
    \sum_{n\ge0}C_k(\lambda,n)C_\ell(\lambda,n)
    e^{-\lambda}\frac{\lambda^n}{n!}
    =\delta_{k,\ell}\frac{k!}{\lambda^k},
\end{equation}
where $\delta_{a,b}$ denotes the Kronecker symbol.
\end{lmm}
For self-containedness and in view of the importance of this
orthogonality relation to our analysis below, we give here a proof
similar to the original one by Jordan \cite{Jordan26}.
\begin{proof}
We start from the expansion
\begin{equation}\label{Exp-Ch-poly}
    C_k(\lambda,n)
    =\sum_{0\le j\le k}\binom{k}{j}(-1)^{k-j}
    \frac{n(n-1)\cdots(n-j+1)}{\lambda^j},
\end{equation}
which follows directly from (\ref{Cklm}). Differentiating both sides
of (\ref{def-char}) $j$ times with respect to $z$ and substituting
$z=1$, we get
\[
    \sum_{n\ge0}e^{-\lambda}\frac{\lambda^n}{n!}
    C_k(\lambda,n)n(n-1)\cdots(n-j+1)
    =\begin{cases}
        j!&\text{if $j=k$};\\
        0&\text{if $j< k$},
    \end{cases}
\]
which means that the Charlier polynomials $C_k(\lambda,x)$ are
orthogonal  to any falling factorials of the form
$x(x-1)\cdots(x-j+1)$ with $j<k$ with respect to the Poisson
measure. Now without loss of generality, we may assume that $\ell\le
k$. Then applying (\ref{Exp-Ch-poly}), we get
\begin{equation*}
\begin{split}
    \sum_{n\ge0} e^{-\lambda}\frac{\lambda^n}{n!}
    C_k(\lambda,n)C_\ell(\lambda,n)
    &=\sum_{0\le j\le \ell}\binom{\ell}{j} (-1)^{\ell-j}
    {\lambda^{-j}}\sum_{n\ge0} e^{-\lambda}\frac{\lambda^n}{n!}
    C_k(\lambda,n){n(n-1)\cdots(n-j+1)}\\
    &=\sum_{0\le j\le \ell}\binom{\ell}{j} (-1)^{\ell-j}
    {\lambda^{-j}}\delta_{k,j}k! \\
    &=\delta_{k,\ell}\frac{k!}{\lambda^k}.
\end{split}
\end{equation*}
This completes the proof.
\end{proof}

\subsection{The Charlier-Parseval identity}

Assume that we have a generating function
\[
    F(z)=\sum_{n\ge0} A_n z^n,
\]
which can be written in the form
\begin{align}\label{Fzfz}
    F(z)=e^{\lambda(z-1)}f(z).
\end{align}
Let
\[
    f(z)=\sum_{j\ge0} a_j(z-1)^j.
\]
Then, by (\ref{Cklm}), we have \emph{formally} the Charlier-Jordan
expansion
\begin{align} \label{An-Ch-exp}
    A_n = e^{-\lambda}\frac{\lambda^n}{n!}
    \sum_{j\ge0} a_j C_j (\lambda,n),
\end{align}
and we expect that $A_n$ will be close to $e^{-\lambda}\lambda^n/n!$
if $f(z)$ is close to $1$, or, alternatively, if $a_0$ is close to
$1$ and all other $a_j$'s are close to $0$. The following identity
provides our first step in quantifying such a heuristic.

\begin{prop}[Charlier-Parseval identity] \label{parseval}
Assume that $f(z)$ is analytic in the whole complex plane and
satisfies
\begin{align}\label{fz-est1}
    |f(z)|=O\left( e^{H|z-1|^2}\right),
\end{align}
as $|z|\to \infty$. Then for any $\lambda>2H$
\begin{align} \label{Ch-Par}
\begin{split}
    \sum_{n\ge0}\left|\frac{A_n}{e^{-\lambda}
    \frac{\lambda^n}{n!}}\right|^2e^{-\lambda}
    \frac{\lambda^n}{n!} =\int_0^\infty
    I(\sqrt{r/\lambda})e^{-r}\,dr,
\end{split}
\end{align}
where
\begin{align}\label{Ir}
    I(r):=\frac{1}{2\pi}\int_{-\pi}^\pi|f(1+re^{it})|^2\,dt.
\end{align}
\end{prop}
\begin{proof}
Since by definition $I(r)=\sum_{j\ge0} |a_j|^2 r^{2j}$ and the
condition (\ref{fz-est1}) implies the convergence of the series
$\sum_{j\ge0} |a_j|^2 j!/\lambda^j$, it follows that
\begin{equation}\label{Parseval}
    \int_0^\infty I\left(\sqrt{r/\lambda} \right)e^{-r}\,dr
    =\sum_{j\ge0}|a_j|^2 \frac{j!}{\lambda^j}.
\end{equation}
Both the series and the integral are convergent because, by
(\ref{fz-est1}), $I(r) = O(e^{2Hr^2})$.

Again by definition
\[
    \sum_{n\ge0} A_n z^n
    =e^{\lambda(z-1)}\sum_{j\ge0} a_j(z-1)^j.
\]
Taking coefficient of $z^n$ on both sides, we obtain
(\ref{An-Ch-exp}), which can be written as
\[
    \frac{A_n}{e^{-\lambda}\frac{\lambda^n}{n!}}
    =\sum_{j\ge0} a_j C_j(\lambda,n),
\]
where the convergence of the above series is pointwise. But the
convergence of the series in (\ref{Parseval}) implies that the
series on the right side also converges in $L_2$-norm with respect
to the Poisson measure $e^{-\lambda} \lambda^n/n!$. Thus the
Proposition follows from (\ref{ortho}).
\end{proof}

In the special cases when $F(z) = (z-1)^ke^{\lambda(z-1)}$, or $A_n=
C_k(\lambda,n) e^{-\lambda}\lambda^n/n!$, we have the identity
\begin{align*}
    \sum_{n\ge0}e^{-\lambda}\frac{\lambda^n}{n!}
    \,|C_k(\lambda,n)|^2 = k!\lambda^{-k}\qquad(k=0,1,\dots),
\end{align*}
which is nothing but (\ref{ortho}) with $k=\ell$. This implies that
\begin{align}\label{id1-CPI}
    \sum_{n\ge0}e^{-\lambda}\frac{\lambda^n}{n!}
    \,|C_k(\lambda,n)| \le \sqrt{k!}\lambda^{-k/2}
    \qquad(k=0,1,\dots).
\end{align}

\subsection{A probabilistic interpretation of the Charlier-Parseval
identity}

Assume that $F(z)$ is a probability generating function of some
non-negative integer valued random variable $X$ having the form
\[
    F(z):=\sum_{m\ge0} \mathbb{P}(X=m)z^m
    = e^{\lambda(z-1)}\sum_{j\ge0} a_j(z-1)^j.
\]
Applying the Charlier-Parseval identity (\ref{Ch-Par}) and
(\ref{Parseval}) to $F$ gives
\[
    \sum_{m\ge0}\left|\frac{\mathbb{P}(X=m)}
    {e^{-\lambda}\frac{\lambda^m}{m!}}-1\right|^2
    e^{-\lambda}\frac{\lambda^m}{m!}
    =\sum_{j\ge1}\frac{j!}{\lambda^j}|a_j|^2,
\]
provided that both series converge. In view of the orthogonality
relations (\ref{ortho}), the coefficients $a_j$ can be expressed as
\[
    a_j=\frac{\lambda^j}{j!}\sum_{m\ge0}
    \mathbb{P}(X=m)C_j(\lambda,m)
    =\frac{\lambda^j}{j!}\mathbb{E}C_j(\lambda,X).
\]
Thus
\[
    \sum_{m\ge0}\left|\frac{\mathbb{P}(X=m)}
    {e^{-\lambda}\frac{\lambda^m}{m!}}-1\right|^2
    e^{-\lambda}\frac{\lambda^m}{m!}
    =\sum_{j\ge1}\frac{\lambda^j}{j!}
    \bigl|\mathbb{E}C_j(\lambda,X)\bigr|^2.
\]
This identity relates the closeness of $X$ to Poisson measure by
means of the moments of $X$ since the quantity
$\mathbb{E}C_j(\lambda,X)$ is a linear combination of the moments of
$X$.

On the other hand, it is also clear, by Cauchy-Schwarz inequality,
that the series on the right-hand side satisfies
\[
    \sum_{j\ge1}\frac{\lambda^j}{j!}
    \bigl|\mathbb{E}C_j(\lambda,X)\bigr|^2
    =\sup \frac{\left(\mathbb{E} \sum_{j\ge1} a_jC_j(\lambda,X)
    \right)^2}{\sum_{j\ge1} a_j^2j!/\lambda^j},
\]
where the supremum is taken over all real sequences
$\{a_j\}_{j\ge1}$ such that $\sum_{j\ge1} a_j^2j!/\lambda^j<\infty$.
Let
\[
    g(x):=\sum_{j\ge1} a_j C_j(\lambda,x).
\]
Then
\[
    \sup \frac{\left(\mathbb{E}
    \sum_{j\ge1} a_jC_j(\lambda,X)
    \right)^2}{\sum_{j\ge1}a_j^2j!/\lambda^j}
    =\sup_{\mathbb{E}g(\zeta)=0}
    \frac{\bigl(\mathbb{E}g(X)\bigr)^2}{\mathbb{E}g(\zeta)^2},
\]
where $\zeta$ is a Poisson random variable with mean $\lambda$.

Applying the difference equation (\ref{diff-eq}) for Charlier
polynomials and taking into account that $a_0=\mathbb{E}g(X)=0$. we
then have
\[
    g(X)=\frac{1}{\lambda}\sum_{j\ge1}
    a_j\mathbb{E}\bigl(XC_{j-1}
    (\lambda,X-1)-\lambda C_{j-1}(\lambda,X)\bigr)
    =\frac{1}{\lambda}\bigl(Xh(X-1)-\lambda h(X)\bigr),
\]
where $h(x)=\sum_{j\ge1}a_kC_{j-1}(\lambda,x)$. Thus we can write
\[
    \Biggl(\sum_{m\ge0}\left|\frac{\mathbb{P}(X=m)}
    {e^{-\lambda}\frac{\lambda^m}{m!}}-1\right|^2
    e^{-\lambda}\frac{\lambda^m}{m!}\Biggr)^{1/2}
    =\sup\mathbb{E}\bigl(Xh(X-1)-\lambda h(X)\bigr),
\]
the supremum being taken over all functions $h$ such that
\(\mathbb{E}\bigl(\zeta h(\zeta-1)-\lambda h(\zeta)\bigr)^2=1\). The
right-hand side of the last expression is reminiscent of the
Chen-Stein equation; see the book \cite{BHJ92}; see also Goldstein
and Reinert \cite{GR05} and the references therein for the
connection between orthogonal polynomials and Stein's method.

\subsection{Asymptotic forms of the Charlier-Parseval
identity}

The identity (\ref{Ch-Par}) can be readily extended to the following
effective (or asymptotic) versions for large $\lambda$.

\begin{prop}[Asymptotic forms of the Charlier-Parseval identity]
Let $F(z)$ and $f(z)$ be defined as above. Assume that $f$ is an
entire function and satisfies the condition
\begin{equation}\label{fz-est2}
    |f(z)|\le K e^{H|z-1|^2},
\end{equation}
for all $z\in \mathbb{C}$, with some positive constants $K$ and $H$.
Then uniformly for all $N\ge 0$ and $\lambda\ge(2+\ve)H$ with
$\ve>0$
\begin{align} \label{Ch-Par1}
    \sum_{n\ge0}\left|\frac{A_n}{e^{-\lambda}\frac{\lambda^n}{n!}}
    -\sum_{0\le j\le N} a_jC_j(\lambda,n)\right|^2 e^{-\lambda}
    \frac{\lambda^n}{n!}
    &\le K^2\frac{2+\ve}{\ve}\left(
    \frac{(2+\ve)H}{\lambda}\right)^{N+1},\\
    \sum_{n\ge0}\left|{A_n}-{e^{-\lambda}
    \frac{\lambda^n}{n!}}\sum_{0\le j\le N}
    a_jC_j(\lambda,n)\right|
    &\le K\sqrt{\frac{2+\ve}{\ve}}\left(
    \frac{(2+\ve)H}{\lambda}\right)^{(N+1)/2},\label{Ch-Par2}
\end{align}
and uniformly for all $n\ge 0$
\begin{align} \label{Ch-Par3}
    \left|A_n-e^{-\lambda}\frac{\lambda^n}{n!}\left(
    \sum_{0\le j\le N} a_jC_{j}(\lambda,n)\right)\right|
    \le K\frac{2+\ve}{\ve}
    \cdot\frac1{\sqrt{\lambda}}\left(
    \frac{(2+\ve)H}{\lambda}\right)^{(N+1)/2}.
\end{align}
\end{prop}
\begin{proof}
Applying (\ref{Ch-Par}) with $\lambda=(2+\ve)H$ and using the upper
bound $I(r)\le K^2e^{2Hr^2}$ (by (\ref{fz-est2})), we get
\begin{align*}
    \sum_{j\ge0} \frac{|a_j|^2j!}{\bigl((2+\ve)H\bigr)^j}
    &=\int_0^\infty I\left(\sqrt{\frac{r}{(2+\ve)H}}
    \right)e^{-r}\,dr \\
    &\le K^2 \int_0^\infty
    e^{-r(1-2/(2+\ve))}\,dr\\
    &= K^2\frac{2+\ve}{\ve}.
\end{align*}
Applying again Proposition~\ref{parseval} but to the function
$f(z)=g(z)-\sum_{0\le j\le N} a_j(z-1)^j$ and using the above
estimate for $\lambda\ge(2+\ve)H$, we get
\[
\begin{split}
    \sum_{n\ge0}\left|\frac{A_n}{e^{-\lambda}
    \frac{\lambda^n}{n!}}-\sum_{0\le j\le N}
    a_jC_j(\lambda,n)\right|^2
    e^{-\lambda}\frac{\lambda^n}{n!}
    &=\sum_{j>N}|a_j|^2\frac{j!}{\lambda^j}\\
    &\le \frac{1}{\lambda^{N+1}} \sum_{j>N}
    \frac{|a_j|^2j!}{\bigl((2+\ve)H\bigr)^{j-(N+1)}}\\
    &=\frac{\bigl((2+\ve)H\bigr)^{N+1}}{\lambda^{N+1}}
    \sum_{j>N} \frac{|a_j|^2j!}{\bigl((2+\ve)H\bigr)^{j}}\\
    &\le K^2\frac{2+\ve}{\ve}
    \left(\frac{(2+\ve)H}{\lambda}\right)^{N+1}.
\end{split}
\]
Thus (\ref{Ch-Par1}) follows and the estimate (\ref{Ch-Par2}) is an
immediate consequence of Cauchy-Schwarz inequality.

For (\ref{Ch-Par3}), we apply Proposition~\ref{parseval} to the
function
\[
    (1-z)\left(f(z)-\sum_{0\le j\le N} a_j(z-1)^j\right),
\]
and obtain
\[
    \sum_{n\ge0}\left|\frac{A_n-A_{n-1}}
    {e^{-\lambda}\frac{\lambda^n}{n!}}-\sum_{0\le j\le N}
    a_jC_{j+1}(\lambda,n)\right|^2
    e^{-\lambda}\frac{\lambda^n}{n!}
    =\sum_{j>N}\frac{|a_{j}|^2(j+1)!}{\lambda^{j+1}}.
\]
By partial summation, (\ref{diff-eq}) and Cauchy-Schwarz inequality
\begin{align}
    \left|A_n-e^{-\lambda}\frac{\lambda^n}{n!}
    \left(\sum_{0\le j\le N}a_jC_{j}
    (\lambda,n)\right)\right|
    &\le\sum_{0\le m\le n}
    \left|{A_m-A_{m-1}}-e^{-\lambda}\frac{\lambda^m}{m!}
    \left(\sum_{0\le j\le N}
    a_jC_{j+1}(\lambda,m)\right)\right|\nonumber \\
    & \le\left(\sum_{m\ge0}
    \left|\frac{A_m-A_{m-1}}{e^{-\lambda}
    \frac{\lambda^m}{m!}}-\sum_{0\le j\le N}
    a_jC_{j+1}(\lambda,n)\right|^2
    e^{-\lambda}\frac{\lambda^m}{m!}\right)^{1/2}\nonumber\\
    & =\left(\sum_{j>N}\frac{|a_{j}|^2(j+1)!}
    {\lambda^{j+1}}\right)^{1/2}. \label{An-bd}
\end{align}
Now for $\lambda\ge(2+\ve)H$
\begin{align*}
    \sum_{n\ge0} \frac{|a_n|^2(n+1)!}{\lambda^{N+1}}
    &=\frac{1}{\lambda}\int_0^\infty
    I\left(\sqrt{r/\lambda} \right)re^{-r}\,dr\\
    &\le\frac{K^2}{\lambda} \int_0^\infty
    e^{-r(1-2H/\lambda)}r\,dr\\
    &=\frac{K^2}{\lambda(1-2H/\lambda)^2}.
\end{align*}
Thus (\ref{Ch-Par3}) follows from substituting this bound into
(\ref{An-bd}).
\end{proof}

\subsection{Some useful estimates of Tauberian type}

We now derive a few other effective bounds for certain partial sums
or series by applying the Charlier-Parseval bounds we derived above;
these bounds are more suitable for use for the diverse Poisson
approximation distances we will consider. They are the types of
results that have more or less the flavor of typical Tauberian
theorems.

Assume that $\zeta_\lambda$ is a Poisson$(\lambda)$ distribution.
Denote by
\[
    Z(n)=\min \left\{\mathbb{P}(\zeta_\lambda\le n),
    \mathbb{P}(\zeta_\lambda>n)\right\}.
\]
It is clear that $Z(n)\le 1/2$.

\begin{prop}\label{int-ineq} Let $F, f, A_n, a_n$ and $I$ be
defined as in (\ref{Fzfz}) and (\ref{Ir}). Assume that $f(z)$ is an
entire function and satisfies the condition (\ref{fz-est1}). Then
for $\lambda>2H$ the following inequalities hold. For $n\ge 0$,
\begin{align}\label{int-ineq-1}
    \sum_{n\ge0}|A_n|
    &\le \left(\int_0^\infty
    I(\sqrt{r/\lambda})e^{-r}\,dr\right)^{1/2}, \\
    \label{int-ineq-2} |A_n|
    &\le \frac{1}{\sqrt{\lambda}}
    \left(\int_0^\infty I(\sqrt{r/\lambda})re^{-r}
    \,dr\right)^{1/2}\sqrt{Z(n)}.
\end{align}

If we additionally assume that $F(1)=0$, then for $n\ge0$,
\begin{align}\label{int-ineq-3}
    \sum_{n\ge0}|A_0+A_1+\cdots+A_n|
    &\le\sqrt{\lambda}\left(\int_0^\infty
    I(\sqrt{r/\lambda} )r^{-1}e^{-r}\,dr\right)^{1/2},\\
    \label{int-ineq-4} |A_0+A_1+\cdots+A_n|
    &\le\left(\int_0^\infty I(\sqrt{r/\lambda})
    e^{-r}\,dr\right)^{1/2}\sqrt{Z(n)}.
\end{align}
\end{prop}
\begin{proof} By Cauchy-Schwarz inequality
\[
    \sum_{n\ge0} |A_n|=\sum_{n\ge0}
    \frac{|A_n|}{e^{-\lambda}\frac{\lambda^n}{n!}}
    \left(e^{-\lambda}\frac{\lambda^n}{n!}\right)^{1/2}
    \left(e^{-\lambda}\frac{\lambda^n}{n!}\right)^{1/2}
    \le \left(\sum_{n\ge0}\left|
    \frac{A_n}{e^{-\lambda}\frac{\lambda^n}{n!}}\right|^2
    e^{-\lambda}\frac{\lambda^n}{n!} \right)^{1/2}.
\]
The upper bound (\ref{int-ineq-1}) then follows from (\ref{Ch-Par}).

The third inequality (\ref{int-ineq-3}) is proved by applying
(\ref{int-ineq-1}) to the function $F_1(z) := F(z)/(1-z)$. Note that
the condition $F(1)=0$ implies that $F_1(z)$ is regular at $z=1$.
With this $F_1$, (\ref{int-ineq-1}) now has the form
\[
    \sum_{n\ge0} |A_0+A_1+\cdots+A_n|
    \le\left(\int_0^\infty I_1(\sqrt{r/\lambda} )
    e^{-r}\,dr\right)^{1/2},
\]
where
\[
    I_1(r)=\frac{1}{2\pi r^2}\int_{-\pi}^\pi
    |f(1+re^{it})|^2\,dt = I(r)/r^2,
\]
and (\ref{int-ineq-3}) follows.

For the fourth inequality (\ref{int-ineq-4}), we start from applying
the Cauchy-Schwarz inequality, giving
\begin{equation}\label{first-bd}
    |A_0+A_1+\cdots+A_n|
    \le \left(\sum_{j\ge0}
    \left|\frac{A_j}{e^{-\lambda}\frac{\lambda^j}{j!}}
    \right|^2e^{-\lambda}\frac{\lambda^j}{j!}
    \right)^{1/2}\left(\sum_{0\le j\le n}
    e^{-\lambda} \frac{\lambda^j}{j!}\right)^{1/2}.
\end{equation}
On the other hand, the condition $F(1)=0$ implies that $\sum_{j\ge0}
A_j=0$. Consequently,
\begin{align}\label{second-bd}
    |A_0+A_1+\cdots+A_n|&=|A_{n+1}+A_{n+2}+\cdots|\notag \\
    &\le \left(\sum_{j\ge0}
    \left|\frac{A_j}{e^{-\lambda}
    \frac{\lambda^j}{j!}}\right|^2
    e^{-\lambda}\frac{\lambda^j}{j!}\right)^{1/2}
    \left(\sum_{j>n}e^{-\lambda}\frac{\lambda^j}{j!}
    \right)^{1/2}.
\end{align}
Taking the minimum of the two upper bounds (\ref{first-bd}) and
(\ref{second-bd}), we obtain (\ref{int-ineq-4}).

Finally, the second inequality (\ref{int-ineq-2}) follows from
(\ref{int-ineq-4}) by applying it to the generating function
$(1-z)F(z)$ instead of $F(z)$.
\end{proof}

\section{Applications. I. Distances for Poisson approximation}
\label{sec:App1}

We apply in this section the diverse tools based on the
Charlier-Parseval identity and derive bounds for the closeness
between the Poisson-binomial distribution and a Poisson distribution
with the same mean. We need a few simple inequalities.

\subsection{Lemmas}
\begin{lmm}\label{1zez} The inequalities
\begin{align}
    |(1+z)e^{-z}| &\le e^{|z|^2/2} \label{ez-ineq}\\
    \left|(1+z)e^{-z} + \sum_{0\le j\le m}
    \frac{j-1}{j!}\,(-z)^j \right|
    &\le c_m |z|^{m+1} e^{|z|^2/2},\label{ez-ineq2}
\end{align}
hold for all $z\in \mathbb{C}$, where $m\ge1$ and
\begin{align} \label{cm}
    c_m :=\frac1{m!}\int_0^1 e^{t^2/2} (1-t)^{m-1}
    (m-1+t) dt.
\end{align}
\end{lmm}
\begin{proof}
Write $z=re^{it}$, where $r>0$ and $t\in\mathbb{R}$. Then, by
$1+x\le e^x$ for $x\in \mathbb{R}$,
\[
\begin{split}
    |(1+z)e^{-z}|&=\sqrt{1+2r\cos t+r^2}\,e^{-r\cos t}\\
    &\le e^{r\cos t +r^2/2-r\cos t} \\
    &=e^{r^2/2}.
\end{split}
\]
For (\ref{ez-ineq2}), we start with the relation
\[
    e^z - \sum_{j<m} \frac{z^j}{j!}
    = \frac{z^m}{(m-1)!}\int_0^1 e^{tz} (1-t)^{m-1}dt,
\]
and deduce that
\[
    (1-z)e^z + \sum_{0\le j\le m} \frac{j-1}{j!}\,z^j
    = -\frac{z^{m+1}}{m!}\int_0^1 e^{tz} (1-t)^{m-1}
    (m-1+t) dt,
\]
for $m\ge1$. Thus (\ref{ez-ineq2}) follows from the inequality
$|tz|\le |z|^2/2+t^2/2$.
\end{proof}

\begin{rem} \label{rem0}
Note that in the proof of (\ref{ez-ineq2}), we have the inequality
\[
    \frac{1+(x-1)e^x}{x^2e^{x^2/2}}\le c_1 = \sqrt{e}-1
    = 0.64872\dots\qquad(x\in\mathbb{R}),
\]
which can easily be sharpened, by elementary calculus, to
\[
    \frac{1+(x-1)e^x}{x^2e^{x^2/2}}\le 0.63236\dots.
\]
But this improvement over $c_1$ is marginal, so we retain the
simpler upper bound $c_1$ in the following use.
\end{rem}

The next lemma is crucial in applying our Charlier-Parseval bounds
derived above.
\begin{lmm}
\label{P0-ineq} The inequality
\begin{equation}\label{Ch-ineq1}
    \left| \prod_{1\le k\le n}(1+v_k)e^{-v_k} -1\right|
    \le c_1 V_2 e^{V_2/2},
\end{equation}
holds for any complex numbers $\{v_k\}$, where
\begin{align}\label{Vm}
    V_m := \sum_{1\le k\le n} |v_k|^m.
\end{align}
\end{lmm}
\begin{proof}
By partial summation
\begin{align} \label{partial-summ}
    \prod_{1\le k\le n} \xi_k - \prod_{1\le k\le n}
    \eta_k = \sum_{1\le k\le n} (\xi_k - \eta_k)
    \left(\prod_{1\le j<k} \xi_j\right)
    \left(\prod_{k<j\le n} \eta_j \right),
\end{align}
for nonzero $\{\xi_k\}$ and $\{\eta_k\}$. Applying this formula, we
get
\[
    \prod_{1\le k\le n}(1+v_k)e^{-v_k} -1
    = \sum_{1\le k\le n}\bigl((1+v_k)e^{-v_k}-1\bigr)
    \prod_{1\le j<k}(1+v_j)e^{-v_j}.
\]
By the two inequalities (\ref{ez-ineq}) and (\ref{ez-ineq2}) with
$m=1$, we then obtain
\[
    \left|\prod_{1\le k\le n}(1+v_k)e^{-v_k} -1\right|
    \le c_1\sum_{1\le k\le n} |v_k|^2
    \prod_{1\le j<k}e^{|v_j|^2/2},
\]
and (\ref{Ch-ineq1}) follows.
\end{proof}

\subsection{New results}

We are ready to apply in this section the tools we developed above
to derive bounds for several Poisson approximation distances.

Let
\[
    S_n:=X_1+X_2+\cdots+X_n,
\]
where the $X_j$'s are independent Bernoulli random variables with
\[
    \mathbb{P}(X_j=1)=1-\mathbb{P}(X_j=0)=p_j
    \qquad(1\le j\le n).
\]
Then, \emph{here and throughout this section},
\begin{align}\label{PB-Fz}
    F(z) := \sum_{0\le m\le n}\mathbb{P}(S_n=m)z^m
    = \prod_{1\le j\le n}(q_j+p_jz),
\end{align}
where $q_j := 1-p_j$. Define $\lambda_m := \sum_{1\le j\le n}
p_j^m$, $\lambda = \lambda_1$ and $\theta := \lambda_2/\lambda_1$.

Let $\Po(\lambda)$ denote a Poisson distribution with mean
$\lambda$.
\begin{thm}\label{thm-P0}
We have the following estimates: $(i)$ for the $\chi^2$-distance
\[
    d_{\chi^2}(\mathscr{L}(S_n),\Po(\lambda))
    :=\sum_{m\ge0}\left|\frac{\mathbb{P}(S_n=m)}
    {e^{-\lambda}\frac{\lambda^m}{m!}}-1\right|^2
    e^{-\lambda}\frac{\lambda^m}{m!}
    \le \frac{2c_1^2\theta^2}{(1-\theta)^{3}};
\]
$($ii$)$ for the total variation distance
\[
    \dtv(\mathscr{L}(S_n),\Po(\lambda))
    :=\frac{1}{2}\sum_{m\ge0}\left|
    \mathbb{P}(S_n=m)-e^{-\lambda}\frac{\lambda^m}{m!}\right|
    \le\frac{c_1\theta}{\sqrt{2}(1-\theta)^{3/2}};
\]
and $($iii$)$ for the Wasserstein (or Fortet-Mourier) distance
\[
    d_W(\mathscr{L}(S_n),\Po(\lambda))
    := \sum_{m\ge0}\Biggl|\mathbb{P}(S_n\le m)
    -\sum_{j\le m}e^{-\lambda}\frac{\lambda^j}{j!}\Biggr|
    \le\frac{c_1\lambda_2}{\sqrt{\lambda}(1-\theta)}.
\]
We also have the following non-uniform bounds for $m\ge0$: $($iv$)$
for the Kolmogorov distance
\[
    \Biggl|\mathbb{P}(S_n\le m)-\sum_{j\le m}
    e^{-\lambda}\frac{\lambda^j}{j!}\Biggr|
    \le \frac{\sqrt{2}c_1\theta}{(1-\theta)^{3/2}}
    \,\sqrt{Z(m)};
\]
and $($v$)$ for the point metric
\[
    \left|\mathbb{P}(S_n=m)
    -e^{-\lambda}\frac{\lambda^m}{m!}\right|
    \le\frac{\sqrt{6}c_1\theta}{(1-\theta)^2\sqrt{\lambda}}
    \,\sqrt{Z(m)}.
\]
\end{thm}
\begin{proof}
For $(i)$, we apply (\ref{Ch-Par}) to the function
$F(z)-e^{\lambda(z-1)}$ and use the inequality (\ref{Ch-ineq1}) with
$v_j=p_jre^{it}$ to estimate the integral $I$. This yields
\begin{align}\label{PB-Ir}
    I(r) &=\frac{1}{2\pi}\int_{-\pi}^\pi
    \left|\prod_{1\le j\le n}(1+p_jre^{it})
    e^{-p_jre^{it}}-1\right|^2\,dt \notag \\
    &\le c_1 ^2\lambda_2^2r^4e^{\lambda_2r^2}
\end{align}
hence
\begin{align*}
    \int_0^\infty I(\sqrt{r/\lambda} )e^{-r}\,dr
    &\le c_1 ^2\theta^2\int_0^\infty
    r^2 e^{-r(1-\theta)}\,dr \\
    &= \frac{2c_1 ^2 \theta^2}{(1-\theta)^{3}},
\end{align*}
and the estimate in $(i)$ for the $\chi^2$-distance follows.

Similarly, the inequalities in $($\emph{ii}$)$ and in
$($\emph{iv}$)$ follow from substituting the estimate (\ref{PB-Ir})
into the two inequalities (\ref{int-ineq-1}) and (\ref{int-ineq-4})
respectively.

As to the non-uniform estimate in $($\emph{v}$)$ for the point
metric, we have, again, by (\ref{PB-Ir}),
\begin{align*}
    \int_0^\infty I(\sqrt{r/\lambda} )re^{-r}\,dr
    &\le c_1 ^2\theta^2\int_0^\infty r^3 e^{-r(1-\theta)}\,dr
    \\ &\le \frac{6c_1 ^2\theta^2}{(1-\theta)^4}.
\end{align*}
Substituting this estimate in (\ref{int-ineq-2}) gives the
inequality in $($\emph{v}$)$.

Finally, the upper bound in $($\emph{iii}$)$ for $d_W$ is derived
similarly by the inequality (\ref{int-ineq-3}) using again
(\ref{PB-Ir})
\[
    \int_0^\infty r^{-1}e^{-r} I(\sqrt{r/\lambda} )\,dr
    \le \frac{c_1 ^2\theta^2}{(1-\theta)^2}.
\]
This completes the proof of the theorem.
\end{proof}


The reason of studying the $\chi^2$-distance (also referred to as
the quadratic divergence) is at least twofold in addition to its
applications in real problems. First, it is structurally simpler
than most other distances because it satisfies the following
identity.
\begin{cor} Let $\{a_j\}$ be given by
\begin{equation}\label{F-e-exp}
    F(z)-e^{\lambda (z-1)}=e^{\lambda(z-1)}
    \sum_{j\ge2} a_j(z-1)^j,
\end{equation}
where $F$ is given in (\ref{PB-Fz}). Then the $\chi^2$-distance
satisfies the identity
\begin{align} \label{id-d-chi2}
    d_{\chi^2}(\mathscr{L}(S_n),\Po(\lambda))
    =\sum_{j\ge2}\frac{j!}{\lambda^j}|a_j|^2.
\end{align}
\end{cor}
\begin{proof}
By (\ref{F-e-exp}), we have
\begin{equation}\label{expansion}
    \mathbb{P}(S_n=m)-e^{-\lambda}\frac{\lambda^m}{m!}
    =e^{-\lambda}\frac{\lambda^m}{m!}\sum_{j\ge2}
    a_jC_j(\lambda,m).
\end{equation}
Then (\ref{id-d-chi2}) follows from (\ref{Parseval}).
\end{proof}
Second, the $\chi^2$-distance is often used to provide bounds for
other distances; see \cite{BV08}. An example is as follows.

\begin{cor} The information divergence (or the Kullback-Leibner
divergence) satisfies
\begin{align} \label{PB-dkl}
    \dkl(\mathscr{L}(S_n),\Po(\lambda))
    :=\sum_{m\ge0} \mathbb{P}(S_n=m)
    \log\left(\frac{\mathbb{P}(S_n=m)}
    {e^{-\lambda}\frac{\lambda^m}{m!}}\right)
    \le \frac{2c_1 ^2\theta^2}{(1-\theta)^{3}}.
\end{align}
\end{cor}
\begin{proof}
Given two sequences of non-negative real numbers $x_j$ and $y_j$
such that
\[
    x_0+x_1+\cdots=1\quad\hbox{and}\quad y_0+y_1+\cdots=1.
\]
By the elementary inequality $\log x\le x-1$, we obtain
\[
    \sum_{n\ge0} y_n\log\frac{y_n}{x_n}
    \le \sum_{n\ge0} y_n
    \left(\frac{y_n}{x_n}-1\right)
    =\sum_{n\ge0}\frac{y_n^2}{x_n}-1
    = \sum_{n\ge0} x_n\left(\frac{y_n}{x_n}-1\right)^2.
\]
Thus $\dkl\le d_{\chi^2}$. Now (\ref{PB-dkl}) follows from applying
this inequality with $x_m=e^{-\lambda}\lambda^m/m!$ and
$y_m=\mathbb{P}(S_n=m)$ and then using the inequality in $(i)$ of
Theorem~\ref{thm-P0}.
\end{proof}

Since $Z(m)\le 1/2$, from the two non-uniform estimates
$($\emph{iv}$)$ and $(v)$ of Theorem~\ref{thm-P0}, we easily obtain
that the Kolmogorov distance satisfies
\[
    d_K(\mathscr{L}(S_n),\Po(\lambda))
    := \sup_m\left|\mathbb{P}(S_n\le m)
    - e^{-\lambda} \sum_{0\le j\le m}
    \frac{\lambda^j}{j!}\right|
    \le \frac{c_1 \theta}{(1-\theta)^{3/2}};
\]
and the point metric is bounded above by
\[
    d_P(\mathscr{L}(S_n),\Po(\lambda))
    := \sup_m\left|\mathbb{P}(S_n= m)
    - e^{-\lambda}\frac{\lambda^m}{m!}\right|
    \le\frac{\sqrt{3} c_1 \theta}{\sqrt{\lambda}\,
    (1-\theta)^{2}}.
\]
Note that the estimate so obtained for the Kolmogorov distance is
worse than that obtained by the simple relation $d_K\le\dtv$ and the
estimate $($\emph{ii}$)$ of Theorem~\ref{thm-P0}.

The quantity $Z(m)$ can be readily bounded above by the following
estimate; see also \cite[p.~259]{BHJ92} or \cite{HZ08}.
\begin{lmm}
\[
    Z(m)\le e^{-(m-\lambda)^2/(2(m+\lambda))}.
\]
\end{lmm}
\begin{proof}
Let $r=m/\lambda$. If $m\ge \lambda$, then
\[
    Z(m)\le \mathbb{P}(\zeta_\lambda\ge m)
    \le r^{-m}e^{\lambda(r-1)}=e^{-\lambda\psi(m/\lambda)},
\]
where $\psi(x) := 1-x+x\log x$. We now prove that
\begin{align}\label{xlogx}
    \psi(x) \ge \frac{(1-x)^2}{2(1+x)}
    \qquad(x>0),
\end{align}
or, equivalently,
\[
    \int_0^x \log(1+t)\dd t \ge \frac{x^2}{2(2+x)}
    \qquad(x>-1).
\]
To prove (\ref{xlogx}), observe first that $\log(1+t) \ge t/(1+t)$
for $t>-1$ since $\int_0^t \log(1+v)\dd v\ge0$. Then
\[
    \int_0^x \log(1+t) \dd t \ge \int_0^x \frac{t}{1+t}
    \, \dd t,
\]
which is bounded below by $x^2/(2(2+x))$ by considering the two
cases $x\ge0$ and $x\in(-1,0]$. Thus, by (\ref{xlogx}),
\[
    Z(m) \le e^{-(m-\lambda)^2/(2(m+\lambda))}.
\]
Similarly, if $m\le \lambda$, then $r<1$, and
\[
    Z(m)\le P(\xi_\lambda\le m)
    \le r^{-m}e^{\lambda(r-1)}=e^{-\lambda\psi(m/\lambda)}
    \le e^{-(m-\lambda)^2/(2(m+\lambda))}.
\]
\end{proof}

\section{Applications. II. Second-order estimates}
\label{sec:App2}

We show in this section that the same approach we developed above
can be readily extended for obtaining higher order estimates. For
simplicity, we consider only the second-order estimates for which we
need only to refine Lemma \ref{P0-ineq}. From the formal expansion
(\ref{expansion}), we expect that
\[
    \mathbb{P}(S_n=m)-e^{-\lambda}\frac{\lambda^m}{m!}
    \approx a_2 e^{-\lambda}\frac{\lambda^m}{m!}
    C_2(\lambda,m) +\text{smaller order terms},
\]
where $a_2 = -\lambda_2/2$, and the error terms for Poisson
approximation would be smaller if we take the term $a_2 e^{-\lambda}
\lambda^m C_2(\lambda,m)/m!$ into account.
\begin{lmm} For any complex numbers $\{v_k\}$, the following
inequality holds
\begin{equation}\label{ineq-P1}
\begin{split}
    \left| \prod_{1\le k\le n}(1+v_k)e^{-v_k}
    -1+\frac{1}{2}\sum_{1\le k\le n}v_k^2\right|
    \le \left(\frac{c_1}{4} V_2^2
    +c_2V_3 \right) e^{V_2/2},
\end{split}
\end{equation}
where $V_m$ is defined in (\ref{Vm}), $c_1 = \sqrt{e}-1$ and (see
(\ref{cm}))
\[
    c_2 = \frac12\int_0^1 e^{t^2/2}(1-t^2) dt
    \approx 0.3706.
\]
\end{lmm}
\begin{proof}
By (\ref{partial-summ}),
\begin{align*}
    \prod_{1\le k\le n} (1+v_k) e^{-v_k} - 1
    +\frac12\sum_{1\le k\le n} v_k^2
    &=\sum_{1\le k\le n}\left(
    (1+v_k)e^{-v_k} - 1+\frac{v_k^2}{2}\right)
    \prod_{1\le j<k} (1+v_j)e^{-v_j}\\
    &\qquad-\frac12 \sum_{1\le k\le n}v_k^2\left(
    \prod_{1\le j<k}(1+v_k)e^{-v_k} - 1\right).
\end{align*}
By (\ref{ez-ineq}), (\ref{ez-ineq2}) with $m=2$ and
(\ref{Ch-ineq1}), we then obtain
\begin{align*}
    \left|\prod_{1\le k\le n} (1+v_k) e^{-v_k} - 1
    +\frac12\sum_{1\le k\le n} v_k^2\right|
    &\le c_2\sum_{1\le k\le n}|v_k|^3 \exp\left(
    \frac12\sum_{1\le j\le k} |v_j|^2\right)\\
    &\qquad+\frac{c_1}2 \sum_{1\le k\le n}|v_k|^2
    \sum_{j<k} |v_j|^2 \exp\left(
    \frac12\sum_{1\le j< k} |v_j|^2\right),
\end{align*}
and (\ref{ineq-P1}) follows.
\end{proof}

For simplicity, let
\[
    P_1(z) := e^{\lambda(z-1)}
    \left(1-\frac{\lambda_2}{2}(z-1)^2\right).
\]
Then
\begin{align} \label{P1z}
    [z^m] P_1(z) &= e^{-\lambda}\frac{\lambda^m}{m!}
    \left(1-\frac{\lambda_2}{2}C_2(m,\lambda)\right),\\
    [z^m] \frac{P_1(z)}{1-z} &= \sum_{j\le m}
    e^{-\lambda}\frac{\lambda^j}{j!}+
    \frac{\lambda_2}{2}C_1(m,\lambda)e^{-\lambda}
    \frac{\lambda^m}{m!} \nonumber,
\end{align}
where $C_1, C_2$ are given in (\ref{C1C2}).

With the inequality (\ref{ineq-P1}) and Proposition \ref{int-ineq},
we can now refine Theorem~\ref{thm-P0} as follows.
\begin{thm}\label{thm-P1} For $\theta<1$, we have the following
second-order estimates for $\chi^2$-, total variation and
Wasserstein distances, respectively,
\[
    \sum_{m\ge0}\frac{\left(\mathbb{P}(S_n=m)
    -[z^m]P_1(z)\right)^2}
    {e^{-\lambda}\frac{\lambda^m}{m!}}
    \le\left(\frac{\sqrt{3}\,c_1  \theta^2}
    {\sqrt{2}(1-\theta)^{5/2}}+
    \frac{\sqrt{6}\,c_2\lambda_3}
    {\lambda^{3/2}(1-\theta)^2}\right)^2,
\]
\[
    \frac12\sum_{m\ge0}\left|\mathbb{P}(S_n=m)-
    [z^m]P_1(z)\right|
    \le \frac{\sqrt{3}\,c_1  \theta^2}
    {2\sqrt{2}(1-\theta)^{5/2}}+
    \frac{\sqrt{3}\,c_2\lambda_3}
    {\sqrt{2}\lambda^{3/2}(1-\theta)^2},
\]
\[
    \sum_{m\ge0}\left|\mathbb{P}(S_n\le m) -
    [z^m]\frac{P_1(z)}{1-z}\right|
    \le \sqrt{\lambda}\left(
    \frac{\sqrt{3}c_1 \theta^2}{2\sqrt{2}(1-\theta)^{2}}
    +\frac{\sqrt{2}\,c_2\lambda_3}{\lambda^{3/2}
    (1-\theta)^{3/2}}\right);
\]
and the second-order non-uniform estimates for Kolmogorov distance
and point metric, respectively,
\begin{align*}
    \left|\mathbb{P}(S_n\le m) -[z^m]\frac{P_1(z)}{1-z}\right|
    \le \sqrt{Z(m)}\left(\frac{\sqrt{3}\,c_1  \theta^2}
    {\sqrt{2}(1-\theta)^{5/2}}+
    \frac{\sqrt{6}\,c_2\lambda_3}
    {\lambda^{3/2}(1-\theta)^2}\right),
\end{align*}
\begin{align*}
    \left|\mathbb{P}(S_n=m)-[z^m]P_1(z)\right|
    \le \sqrt{\frac{Z(m)}{\lambda}}
    \left(\frac{\sqrt{15}\,c_1 \theta^2}{\sqrt{2}(1-\theta)^{3}}
    +\frac{2\sqrt{6}\, c_2\lambda_3}
    {\lambda^{3/2}(1-\theta)^{5/2}}\right).
\end{align*}
\end{thm}
\begin{proof}
Let
\[
    F(z)=\prod_{1\le j\le n}(1+p_j(z-1))-e^{\lambda(z-1)}
    \left(1-\frac{\lambda_2}{2}(z-1)^2\right).
\]
Take $v_j=p_j(z-1)$ in inequality (\ref{ineq-P1}). Then
\[
    \left|\prod_{1\le j\le n}(1+p_j(z-1))e^{-p_j(z-1)}
    -1-\frac{\lambda_2}{2}(z-1)^2\right|
    \le \left(\frac{c_1 }{4}\lambda_2^2|z-1|^4+
    c_2 \lambda_3|z-1|^3\right)
    e^{\frac{\lambda_2}{2}|z-1|^2}.
\]
It follows that
\begin{align}\label{Ir-2nd}
    I(r)\le \left(\frac{c_1 }{4}\lambda_2^2r^4
    +c_2\lambda_3r^3\right)^2 e^{{\lambda_2}r^2}.
\end{align}
Substituting this upper bound into the identity (\ref{Ch-Par}) and
using the relation (\ref{P1z}), we obtain
\[
\begin{split}
    &\left(\sum_{m\ge0}\frac{\left(\mathbb{P}(S_n=m)
    -[z^m]P_1(z)\right)^2}
    {e^{-\lambda}\frac{\lambda^m}{m!}}\right)^{1/2} \\
    &\qquad\le \left(\int_0^\infty\left(\frac{c_1 }{4}\theta^2r^2
    +\frac{c_2\lambda_3}{\lambda^{3/2}}r^{3/2}\right)^2
    e^{-(1-\theta)r}\,dr\right)^{1/2}\\
    &\qquad\le\frac{c_1 }{4}\theta^2\left(\int_0^\infty
    r^4e^{-(1-\theta)r}\,dr\right)^{1/2}
    +\frac{c_2\lambda_3}{\lambda^{3/2}}
    \left(\int_0^\infty r^3e^{-(1-\theta)r}\,dr\right)^{1/2}\\
    &\qquad=\frac{c_1 }{4}\theta^2\cdot
    \frac{\sqrt{24}}{(1-\theta)^{5/2}}
    +\frac{c_2\lambda_3}{\lambda^{3/2}}\cdot\frac{\sqrt{6}}
    {(1-\theta)^2},
\end{split}
\]
where we used the Minkowsky inequality. This proves the second-order
estimate for the $\chi^2$-distance.

Similarly, the corresponding estimates for the total variation
distance and the (non-uniform estimate of the) Kolmogorov distance
follow from (\ref{Ir-2nd}) and the two inequalities
(\ref{int-ineq-1}) and (\ref{int-ineq-4}), respectively.

For the point metric, we have, using again (\ref{Ir-2nd}) and the
inequality (\ref{int-ineq-2}),
\[
\begin{split}
    &\sqrt{\frac{\lambda}{Z(m)}}
    \left|\mathbb{P}(S_n=m)-[z^m]P_1(z)\right|\\
    &\qquad\le \left(\int_0^\infty
    \left(\frac{c_1 }{4}\theta^2r^2
    +\frac{c_2\lambda_3}{\lambda^{3/2}}r^{3/2}
    \right)^2re^{-(1-\theta)r}\,dr\right)^{1/2}\\
    &\qquad\le \frac{c_1 }{4}\theta^2
    \left(\int_0^\infty r^5e^{-(1-\theta)r}\,dr\right)^{1/2}
    +\frac{c_2\lambda_3}{\lambda^{3/2}}\left(\int_0^\infty
    r^4e^{-(1-\theta)r}\,dr\right)^{1/2}\\
    &\qquad=\frac{\sqrt{15}\,c_1 \theta^2}{\sqrt{2}(1-\theta)^{3}}
    +\frac{2\sqrt{6}\,c_2\lambda_3}
    {\lambda^{3/2}(1-\theta)^{5/2}}.
\end{split}
\]
Finally, the second-order estimate for the Wasserstein distance
follows from (\ref{Ir-2nd}) and the inequality (\ref{int-ineq-3})
\[
\begin{split}
    &\lambda^{-1/2}
    \sum_{m\ge0}\left|\mathbb{P}(S_n\le m)
    -[z^m]\frac{P_1(z)}{1-z}\right|\\
    &\qquad \le \left(\int_0^\infty
    \left(\frac{c_1 }{4}\theta^2r^2+\frac{c_2\lambda_3}
    {\lambda^{3/2}}r^{3/2}\right)^2r^{-1}
    e^{-(1-\theta)r}\,dr\right)^{1/2}\\
    &\qquad\le \frac{c_1 }{4}\theta^2
    \left(\int_0^\infty r^3e^{-(1-\theta)r}\,dr\right)^{1/2}
    +\frac{c_2\lambda_3}
    {\lambda^{3/2}}\left(\int_0^\infty r^2e^{-(1-\theta)r}
    \,dr\right)^{1/2}.
\end{split}
\]
\end{proof}

\begin{cor} The total variation distance between the distribution of
$S_n$ and a Poisson distribution of mean $\lambda$ satisfies, for
$\theta<1$,
\begin{equation}\label{dtv-P1}
    \dtv(S_n,\mathcal{P}(\lambda))
    \le \frac{\theta}{2^{3/2}}+
    \frac{\sqrt{3}\,c_1  \theta^2}
    {2\sqrt{2}(1-\theta)^{5/2}}+
    \frac{\sqrt{3}\,c_2\lambda_3}
    {\sqrt{2}\lambda^{3/2}(1-\theta)^2}.
\end{equation}
\end{cor}
\begin{proof}
By (\ref{id1-CPI}) with $k=2$, we have
\[
    \frac12\sum_{m\ge0} e^{-\lambda}
    \frac{\lambda^m}{m!}|C_2(\lambda,m)|
    \le\frac{1}{\sqrt{2}\lambda},
\]
and (\ref{dtv-P1}) follows from the second-order estimate for the
total variation distance in Theorem~\ref{thm-P1}.
\end{proof}

\begin{rem}
One can easily derive, by the difference equation (\ref{diff-eq}) of
Charlier polynomials with $k=1$, that (see for example
\cite{Hwang99})
\[
    \frac12\sum_{m\ge0} e^{-\lambda}
    \frac{\lambda^m}{m!}|C_2(\lambda,m)|
    =e^{-\lambda} \left(\frac{\lambda^{m_+-1}}{m_+!}
    (m_+-\lambda) + \frac{\lambda^{m_--1}}{m_-!}(\lambda-m_-)
    \right),
\]
where $m_\pm := \lfloor \lambda+\frac12\pm
\sqrt{\lambda+\frac14}\rfloor$. Asymptotically, for large $\lambda$,
\[
    \frac12\sum_{m\ge0} e^{-\lambda}\frac{\lambda^m}{m!}
    |C_2(\lambda,m)| = \frac{\sqrt{2}}{\sqrt{\pi e}\, \lambda}
    \left(1+O\left(\lambda^{-1}\right)\right).
\]
By a detailed calculus, Roos \cite{Roos01} showed that
\begin{align}\label{Roos-ineq}
    \frac12\sum_{m\ge0} e^{-\lambda}
    \frac{\lambda^m}{m!} |C_2(\lambda,m)|
    \le \frac{3}{2e\lambda},
\end{align}
where numerically
\[
    \left\{\frac1{\sqrt{2}},\frac{3}{2e},
    \frac{\sqrt{2}}{\sqrt{\pi e}}\right\}
    \approx \left\{0.707, 0.552,0.484\right\}.
\]
Of course, we can apply Roos's inequality (\ref{Roos-ineq}) and
replace the constant $1/2^{3/2}\approx 0.354\ldots$ by
$3/(4e)\approx 0.276\ldots$ in the first term of our inequality
(\ref{dtv-P1}).
\end{rem}

\begin{cor} The $\chi^2$-distance satisfies
\begin{align}\label{cor-d-chi2}
    d_{\chi^2}(\mathscr{L}(S_n),\Po(\lambda))
    = \frac{\theta^2}{2}\left(1+O\left(
    \frac{\theta}{(1-\theta)^{5}}\right)\right).
\end{align}
\end{cor}
\begin{proof}
Note that
\[
    0\le\sum_{m\ge0}\frac{\left(\mathbb{P}(S_n=m)
    -[z^m]P_1(z)\right)^2}
    {e^{-\lambda}\frac{\lambda^m}{m!}}
    =\sum_{m\ge0}\frac{\left(\mathbb{P}(S_n=m)
    -{e^{-\lambda}\frac{\lambda^m}{m!}}\right)^2}
    {e^{-\lambda}\frac{\lambda^m}{m!}}-\frac{\theta^2}{2}.
\]
This identity together with the first estimate of Theorem
\ref{thm-P1} and an observation that $\lambda_3\le \lambda_2^{3/2} $
yields (\ref{cor-d-chi2}).
\end{proof}

\begin{rem}
An alternative way to prove (\ref{cor-d-chi2}) is to use the
identity (\ref{id-d-chi2}) and apply the estimate for the
coefficients $a_j$ derived in Shorgin \cite{Shorgin77}
\begin{equation}\label{shorgins-bd}
    |a_j|\le \left(\frac{e\lambda_2}{j}\right)^{j/2}
    \qquad(j\ge2),
\end{equation}
and obtain
\begin{align*}
    \sum_{j\ge3}\frac{j!}{\lambda^j}|a_j|^2
    \le \sum_{j\ge3}  j! (e/j)^j\theta^j
    = O\left(\sum_{j\ge3}  j^{1/2}\theta^j\right),
\end{align*}
by Stirling's formula $j!=O(j^{1/2} (j/e)^j)$, $j\ge1$. This and
$a_2=-\lambda_2/2$ give
\begin{align}\label{d-chi-1}
    d_{\chi^2}(\mathscr{L}(S_n),\Po(\lambda))
    = \frac{\theta^2}{2}\left(1+O\left(
    \frac{\theta}{(1-\theta)^{3/2}}\right)\right).
\end{align}
\end{rem}
For a further refinement of (\ref{cor-d-chi2}), see Corollary
\ref{d-chi2-2}. Note that (\ref{d-chi-1}) implies that
\[
    \dkl(\mathscr{L}(S_n),\Po(\lambda)) \le
    \frac{\theta^2}2\left(1+O\left(
    \frac{\theta}{(1-\theta)^{3/2}}\right)\right).
\]

\section{Applications. III. Approximations by signed measures}
\label{sec:App3}

Since the probability generating function of $S_n$ can be
represented as
\[
    \mathbb{E}z^{S_n} = \exp\left(\sum_{j\ge1}
    \frac{(-1)^{j-1}}{j}\,\lambda_j(z-1)^j\right),
\]
it is well-known since Herrmann \cite{Herrmann65} that smaller error
terms can be achieved if we use finite number of terms in the
exponent to approximate $\mathbb{E}z^{S_n}$; namely,
\[
    \mathbb{E}z^{S_n} \approx \exp\left(\sum_{1\le j\le k}
    \frac{(-1)^{j-1}}{j}\,\lambda_j(z-1)^j\right),
\]
for $k\ge1$. Anther advantage of such approximations is that the
remainder terms tend to zero not only when $\theta\to0$ but also
when $\lambda\to \infty$ (while $\theta$ remaining, say less than
$1-\ve$, $\ve>0$ being a small number). This gives rise to Poisson
approximation via signed measures (sometimes also referred to as
compound Poisson approximations); see Cekanavicius
\cite{Cekanavicius97}, Roos \cite{Roos04}, Barbour et al.\
\cite{BCX07} for more information.

Although these approximations are not probability generating
functions for $k\ge2$, they can numerically and asymptotically be
readily computed. Indeed, for $k=2$
\[
    [z^m]e^{\lambda(z-1)-\lambda_2(z-1)^2/2}
    = e^{-\lambda-\lambda_2/2}
    \frac{\lambda_2^{m/2}}{m!}\,H_m\left(
    \frac{\lambda+\lambda_2}{\sqrt{\lambda_2}}\right),
\]
where the $H_m(x)$'s are the Hermite polynomials.

\subsection{Approximation by
$e^{\lambda(z-1)-\lambda_2(z-1)^2/2}$}

We consider the simplest case of such forms when $k=2$.
\begin{lmm} The inequality
\begin{align}\label{ineq-P2}
    \left|\prod_{1\le k\le n} (1+v_k) e^{-v_k}
    -\exp\left(-\frac12\sum_{1\le k\le n} v_k^2\right)\right|
    \le \left( c_2 V_3 + \frac18V_4\right) e^{V_2/2}
\end{align}
holds for any complex numbers $\{v_k\}$, where $V_m$ is given in
(\ref{Vm}) and $c_2$ in (\ref{cm}).
\end{lmm}
\begin{proof}
Again by (\ref{partial-summ}),
\begin{align*}
    &\prod_{1\le k\le n} (1+v_k) e^{-v_k}
    -\prod_{1\le k\le n} e^{-v_k^2/2} \\
    &\qquad=\sum_{1\le k\le n}
    \left((1+v_k)e^{-v_k} - e^{-v_k^2/2}\right)
    \left(\prod_{1\le j<k} (1+v_j)e^{-v_j}\right)
    \left(\prod_{k<j\le n} e^{-v_j^2/2} \right).
\end{align*}
Now
\begin{align*}
    \left|(1+z)e^{-z}- e^{-z^2/2}\right|
    &= \left|(1+z)e^{-z} -1+\frac{z^2}2
    -\left(e^{-z^2/2}-1+\frac{z^2}2\right)\right|\\
    &=\left|-\frac{z^3}{2} \int_0^1 (1-t^2) e^{-tz} dt
    - \frac{z^4}{4} \int_0^1 (1-t) e^{-tz^2/2} dt\right|\\
    &\le c_2 |z|^3 e^{|z|^2/2} + \frac{|z|^4}{8} e^{|z|^2/2}.
\end{align*}
This and the inequality (\ref{ez-ineq}) yield (\ref{ineq-P2}).
\end{proof}

Let
\[
    P_2(z) := e^{\lambda(z-1)-\lambda_2(z-1)^2/2}.
\]
\begin{thm}\label{thm-P2}
Assume that $\theta < 1$. Then
\begin{align*}
    \sum_{m\ge0}\frac{\left(\mathbb{P}(S_n=m)
    -[z^m]P_2(z)\right)^2}{e^{-\lambda}\frac{\lambda^m}{m!}}
    &\le \frac{\lambda_3^2}{\lambda^{3}}
    \left(\frac{\sqrt{6}\,c_2}{(1-\theta)^2}
    +\frac{\sqrt{3\theta}}{2\sqrt{2}(1-\theta)^{5/2}}\right)^2,\\
    \sum_{m\ge0}\left|\mathbb{P}(S_n=m)-[z^m]P_2(z)\right|
    &\le\frac{\lambda_3}{\lambda^{3/2}}\left(
    \frac{\sqrt{6}\,c_2}{(1-\theta)^2}
    +\frac{\sqrt{3\theta}}{2\sqrt{2}(1-\theta)^{5/2}}\right),\\
    \sum_{m\ge0}\left|\mathbb{P}(S_n\le m)
    -[z^m]\frac{P_2(z)}{1-z}\right|
    &\le \frac{\lambda_3}{\lambda}\left(
    \frac{\sqrt{2}\,c_2}{(1-\theta)^{3/2}}
    +\frac{\sqrt{3\theta}}{4\sqrt{2}(1-\theta)^{2}}\right),\\
    \left|\mathbb{P}(S_n\le m)-[z^m]\frac{P_2(z)}{1-z}\right|
    &\le \frac{\lambda_3}{\lambda^{3/2}}\sqrt{Z(m)}\left(
    \frac{\sqrt{6}\,c_2}{(1-\theta)^2}
    +\frac{\sqrt{3\theta}}{2\sqrt{2}(1-\theta)^{5/2}}\right),\\
    \left|\mathbb{P}(S_n=m)-[z^m]P_2(z)\right|
    &\le \frac{\lambda_3}{\lambda^{2}}\sqrt{Z(m)}
    \left(\frac{2\sqrt{6}\,c_2}{(1-\theta)^{5/2}}
    +\frac{\sqrt{15\theta}}{2\sqrt{2}(1-\theta)^{3}}\right).
\end{align*}
\end{thm}
\begin{proof}
All estimates follow similarly as the proof of Theorem~\ref{thm-P0}
but with
\[
    F(z)=\prod_{1\le j\le n} (1+p_j(z-1))
    -e^{\lambda(z-1)-\lambda_2(z-1)^2/2}.
\]
For the first two estimates of the theorem, we apply the inequality
(\ref{ineq-P2}), which gives
\[
    I(r)\le\left(c_2\lambda_3r^3+\frac{1}{8}\lambda_4r^4\right)^2
    e^{\lambda_2r^2}.
\]
By the inequality $\lambda_4\le \lambda_3\sqrt{\lambda_2}$, we
obtain
\[
\begin{split}
    \left(\sum_{m\ge0}\frac{\left(\mathbb{P}(S_n=m)
    -[z^m]P_2(z)\right)^2}{e^{-\lambda}
    \frac{\lambda^m}{m!}}\right)^{1/2}
    &\le \frac{\lambda_3}{\lambda^{3/2}}\left(\int_0^\infty
    \left(c_2r^{3/2}+\frac{\sqrt{\theta}}{8}r^2\right)^2
    e^{-(1-\theta)r^2}\right)^{1/2}\\
    &\le\frac{\lambda_3}{\lambda^{3/2}}
    \left(\frac{c_2\sqrt{6}}{(1-\theta)^2}
    +\frac{\sqrt{24\theta}}{8(1-\theta)^{5/2}}\right).
\end{split}
\]
Then we apply Proposition \ref{int-ineq}. The other estimates are
similarly proved.
\end{proof}

\begin{lmm} For any $\theta<1$, we have
\begin{align} \label{id-P2}
    \sum_{m\ge0}\frac{\left(e^{-\lambda}\frac{\lambda^m}{m!}
    -[z^m]P_2(z)\right)^2}{e^{-\lambda}\frac{\lambda^m}{m!}}
    =\frac{1}{\sqrt{1-\theta^2}}-1.
\end{align}
\end{lmm}
\begin{proof}
Applying (\ref{Ch-Par}) and (\ref{Parseval}) to the function
\[
    F(z)=e^{\lambda(z-1)}-P_2(z)
    =e^{\lambda(z-1)}\left(\sum_{k\ge1}
    \left(\frac{\lambda_2}{2}\right)^k
    \frac{(z-1)^{2k}}{k!}\right),
\]
we obtain
\[
\begin{split}
    \sum_{m\ge0} \frac{\left(e^{-\lambda}\frac{\lambda^m}{m!}
    -[z^m]P_2(z)\right)^2}{e^{-\lambda}\frac{\lambda^m}{m!}}
    &=\sum_{k\ge1}\left(\frac{\theta}{2}\right)^{2k}
    \frac{(2k)!}{(k!)^2}
    =\frac{1}{\sqrt{1-\theta^2}}-1.
\end{split}
\]
\end{proof}

\begin{cor} For $\theta<1$,
\begin{align} \label{d-chi2-2}
    \left|\left(\sum_{m\ge0} \frac{\left(\mathbb{P}(S_n=m)
    -e^{-\lambda}\frac{\lambda^m}{m!}\right)^2}
    {e^{-\lambda}\frac{\lambda^m}{m!}}\right)^{1/2}-
    \left(\frac{1}{\sqrt{1-\theta^2}}-1\right)^{1/2}\right|
    \le \frac{\lambda_3}{\lambda^{3/2}}
    \left(\frac{c_2\sqrt{6}}{(1-\theta)^2}
    +\frac{\sqrt{24\theta}}{8(1-\theta)^{5/2}}\right).
\end{align}
\end{cor}
\begin{proof}
By applying the Minkowsky inequality and the first estimate of
Theorem~\ref{thm-P2}, we obtain
\begin{align*}
    &\left|\left(\sum_{m\ge0} \frac{\left(\mathbb{P}(S_n=m)
    -e^{-\lambda}\frac{\lambda^m}{m!}\right)^2}
    {e^{-\lambda}\frac{\lambda^m}{m!}}\right)^{1/2}-
    \left(\sum_{m\ge0}\frac{\left(
    e^{-\lambda}\frac{\lambda^m}{m!}
    -[z^m]P_2(z)\right)^2}{e^{-\lambda}
    \frac{\lambda^m}{m!}}\right)^{1/2}\right|\\
    &\qquad \le \left(\sum_{m\ge0}\frac{\left(\mathbb{P}(S_n=m)
    -[z^m]P_2(z)\right)^2}
    {e^{-\lambda}\frac{\lambda^m}{m!}}\right)^{1/2}\\
    &\qquad \le\frac{\lambda_3}{\lambda^{3/2}}
    \left(\frac{c_2\sqrt{6}}{(1-\theta)^2}
    +\frac{\sqrt{24\theta}}{8(1-\theta)^{5/2}}\right).
\end{align*}
Consequently, by (\ref{id-P2}), we obtain (\ref{d-chi2-2}).
\end{proof}

Note that (\ref{d-chi2-2}) implies that, for all $\theta<1$,
\[
    d_{\chi^2}(\mathscr{L}(S_n),\Po(\lambda))
    =\left(\frac{1}{\sqrt{1-\theta^2}}-1\right)
    \left(1+O\left(\frac{\lambda_3}{\lambda_2
    \sqrt{\lambda}(1-\theta)^5}\right)\right).
\]

On the other hand, by the inequality $d_{\chi^2}\ge 4\dtv^2$ (which
following from (\ref{Parseval}) and (\ref{int-ineq-1})), we obtain
another upper bound for $\dtv$.
\begin{cor} For $\theta<1$,
\[
    \dtv(S_n,\mathcal{P}(\lambda))
    \le\frac{1}{2}\left(\frac{1}{\sqrt{1-\theta^2}}
    -1\right)^{1/2}+\frac{\lambda_3}{\lambda^{3/2}}
    \left(\frac{c_2\sqrt{6}}{2(1-\theta)^2}
    +\frac{\sqrt{24\theta}}{16(1-\theta)^{5/2}}\right).
\]
\end{cor}

\section{Comparative discussions}
\label{sec:cd}

We review briefly some known results in the literature and compare
them in this section. For simplicity, we write $d_{*}$ for
$d_{*}(\mathscr{L}(S_n),\Po(\lambda))$ throughout this section,
where $d_{*}$ represents one of the distances we discuss.

Among the five measures of closeness of Poisson approximation
$\{d_{\chi^2}, \dtv, d_W, d_K, d_P\}$, the estimation of the three
$\{d_{\chi^2}, d_K, d_P\}$ is generally simpler in complexity since
they can all be easily bounded above by explicit summation or
integral representations: see (\ref{id-d-chi2}) for $d_{\chi^2}$,
(\ref{dk-ir}) for $d_K$ and (\ref{dp-ir}) for $d_P$.

In addition to the Poisson approximations to $\mathscr{L}(S_n)$ we
consider in this paper, many other different types of approximations
to $\mathscr{L}(S_n)$ were proposed in the literature; these include
Poisson with different mean, compound Poisson, translated Poisson,
large deviations, other perturbations of Poisson, binomial, compound
binomial, etc. They are too numerous to be listed and compared here;
see, for example, Barbour et al.\ \cite{BHJ92}, Roos
\cite{Roos99,Roos07}, Barbour and Chryssaphinou \cite{BC01}, Barbour
and Chen \cite{BC05}, R\"ollin \cite{Rollin07} and the references
therein.

\subsection{The $\chi^2$-distance and the Kullback-Leibner
divergence}

Borisov and Vorozhe\v\i kin \cite{BV08} showed that $d_{\chi^2}\sim
\theta^2/2$ under the assumption that $\theta=o(\lambda^{-1/7})$.
They also derived in the same paper the identity (\ref{id-d-chi2})
in the special case when all $p_j$'s are equal. More refined
estimates were then given. The estimate (\ref{cor-d-chi2}) we
obtained is more general and stronger.

The Kullback-Leibner divergence has been widely studied in the
information-theoretic literature and many results are known. The
connection between $\dtv$ and $\dkl$ for general distributions also
received much attention since they can be used to bridge results in
probability theory and in information theory; see the survey paper
Fedotov et al.\ \cite{FHT03} for more information and references.
One such tool studied is Pinsker's inequality $\dtv\le
\sqrt{\dkl/2}$ (see \cite{FHT03}). Note that in the case of $S_n$,
this inequality implies that $\dtv\le \sqrt{d_{\chi^2}/2}$, while we
have $\dtv\le \sqrt{d_{\chi^2}}/2$ by (\ref{Parseval}) and
(\ref{int-ineq-1}).

Kontoyiannis et al.\ \cite{KHJ05} recently proved, by an
information-theoretic approach, that
\[
    \dkl \le \frac{1}{\lambda}\sum_{1\le j\le n}
    \frac{p_j^3}{1-p_j}.
\]
The right-hand side in the above inequality is, by Cauchy-Schwarz
inequality, always larger than $\theta^2$, provided that at least
one of the $p_j$'s is nonzero, and can be considerably larger than
our estimate (\ref{PB-dkl}) for certain cases. Indeed, take for
example $p_j=1/\sqrt{j+1}$. Then
\[
    \dkl \le \frac{1}{\lambda}\sum_{1\le j\le n}
    \frac{p_j^3}{1-p_j}\asymp \frac{1}{\sqrt{n}},
\]
where the symbol ``$a_n \asymp b_n$" means that $a_n$ is
asymptotically of the same order as $b_n$. Our result (\ref{PB-dkl})
yields in this case the estimate
\[
    \dkl \le \frac{2c_1^2\theta^2}{(1-\theta)^3}
    \asymp \frac{\log^2 n}{n}.
\]

\subsection{The total variation distance}

We mentioned in Introduction some results in Le Cam \cite{LeCam60}
and other refinements in the literature of the form $\dtv\le
c\theta$. We briefly review and compare here other results for
$\dtv$.

\paragraph{First- and second-order estimates.} Kerstan
\cite{Kerstan64}, in addition to proving that $\dtv\le 0.6\theta$
(which was later on corrected to $1.05$ by Barbour and Hall
\cite{BH84}), he also proved the second-order estimate
\[
    \sum_{j\ge0} \left|\mathbb{P}(S_n=j)
    -e^{-\lambda}\frac{\lambda^j}{j!}
    \left(1-\frac{\lambda_2}{2}C_2(\lambda,j)\right)
    \right| \le 1.3\frac{\lambda_3}{\lambda}+3.9\theta^2.
\]
Similar estimates were derived later in Herrmann \cite{Herrmann65},
Chen \cite{Chen75}, Barbour and Hall \cite{BH84}. The order of the
error terms is however not optimal for large $\lambda$; see
Theorem~\ref{thm-P1}.

Many fine estimates were obtained in the series of papers by
Deheuvels, Pfeifer and their co-authors. In particular, Deheuvels
and Pfeifer \cite{DP88} proved $\dtv\le \theta/(1-\sqrt{2\theta})$
for $\theta<1/2$ and the second-order estimate
\[
    \sum_{j\ge0} \left|\mathbb{P}(S_n=j)
    -e^{-\lambda}\frac{\lambda^j}{j!}
    \left(1-\frac{\lambda_2}{2}C_2(\lambda,j)\right)
    \right| \le \frac{(2\theta)^{3/2}}{1-\sqrt{2\theta}},
\]
for $\theta<1/2$, the order of the error terms being tight. For many
other estimates (including higher-order ones), see
\cite{DP88,DPR89}. Their approach is based on a semi-group
formulation, followed by applying the fine estimates of Shorgin
\cite{Shorgin77}, which in turn were obtained by the
complex-analytic approach of Uspensky \cite{Uspensky31}. Following a
similar approach, Witte \cite{Witte90} gives an upper bound of the
form
\[
    \dtv \le \frac{e^{2p_*}\theta}
    {\sqrt{2\pi}(1-2e^{2p_*}\theta)},
\]
for $\theta<\tfrac12 e^{-2p_*}$, as well as other more complicated
ones. Another very different form for $\dtv$ can be found in Weba
\cite{Weba99}, which results from combining several known estimates.

By refining further Deheuvels and Pfeifer's approach, Roos
\cite{Roos99,Roos01} deduced several precise estimates for $\dtv$
and other distances. In particular, he showed that
\[
    \dtv\le \left(\frac3{4e}+\frac{7(3-2\sqrt{\theta})}
    {6(1-\sqrt{\theta})^2}\,\sqrt{\theta}\right)\theta,
\]
when $\theta<1$; see \cite{Roos01} and the references therein. The
proof of this estimate is based on a second-order approximation; see
(\ref{Roos-ineq}).

Note that since $\dtv\le 1$, any result of the form $\dtv\le
\varphi(\theta)\theta$ for $\theta\le \theta_1$, $\theta_1\in(0,1)$,
also leads to an upper bound of the form $\dtv\le c\theta$, where
\[
    c = \sup_{0\le t\le \theta_0} \varphi(t),
\]
$\theta_0 := \min\{\theta_1,\theta_2\}$, $\theta_2\in(0,1)$ solving
the equation $t\varphi(t)=1$.

Higher-order approximations based on Charlier expansion are studied
in Herrmann \cite{Herrmann65}, Barbour \cite{Barbour87}, Deheuvels
and Pfeifer \cite{DP88}, Barbour et al.\ \cite{BHJ92}, Roos
\cite{Roos99,Roos04}.

\paragraph{Approximations by signed measures.} Herrmann
\cite{Herrmann65} proved that, when specializing to the case of
$S_n$,
\[
    \sum_{m\ge0}\left|\mathbb{P}(S_n=m)-
    [z^m]e^{\lambda(z-1)-\lambda_2(z-1)^2/2}\right|
    = O\left(\frac{\lambda_3} {\lambda}\right),
\]
the rate being $\lambda^{1/2}$ away from optimal; see
Theorem~\ref{thm-P2}. Presman \cite{Presman83} considered the
binomial case and derived an optimal error bound. Kruopis
\cite{Kruopis86} extended further Presman's analysis and derived
\begin{align*}
    &\sum_{m\ge0}\left|\mathbb{P}(S_n=m)-
    [z^m]e^{\lambda(z-1)-\lambda_2(z-1)^2/2}\right|\\
    &\qquad \le 10\varpi  \lambda_3 \min\left\{
    1.2\sigma^{-3}+4.2\lambda_2\sigma^{-6},
    2+\sigma^{2}+3.4\lambda_2\right\},
\end{align*}
where $\sigma := \sqrt{\lambda-\lambda_2}$ and
\begin{align} \label{varpi}
     \varpi := \max_{1\le j\le n}
     \sup_{0\le t\le 1} e^{2p_j t(1-p_jt)},
\end{align}
which was in turn refined by Borovkov \cite{Borovkov89}. Hipp
\cite{Hipp86} discussed similar expansions for compound Poisson
distributions and attributed the idea to Kornya \cite{Kornya83}, but
his bounds are weaker for large $\lambda$ in the special case of
$S_n$; see also \v Cekanavi\v cius \cite{Cekanavicius97}. Barbour
and Xia \cite{BX99} proved, as a special case of their general
results, that
\[
    \sum_{m\ge0}\left|\mathbb{P}(S_n=m)-
    [z^m]e^{\lambda(z-1)-\lambda_2(z-1)^2/2}\right|
    \le \frac{4\lambda_3}{\lambda^{3/2}(1-2\theta)
    \sqrt{1-\theta-\max_j p_j(1-p_j)/\lambda}},
\]
when $\theta<1/2$. An extensive study was carried out by \v
Cekanavi\v cius in a series of papers dealing mainly with
Kolmogorov's problem of approximating convolutions by infinitely
divisible distributions; see \v Cekanavi\v cius
\cite{Cekanavicius97,Cekanavicius04} and the references cited there.
Approximation results using signed compound measures under more
general settings than $S_n$ are derived in Borovkov and Pfeifer
\cite{BP96}, Roos \cite{Roos04,Roos07} and \v Cekanavi\v cius
\cite{Cekanavicius04}, Barbour et al.\ \cite{BCX07}.

\paragraph{Other uniform asymptotic approximations.} The estimate
$\dtv\sim \theta/\sqrt{2\pi e}$ holds whenever $\theta\to0$. A
uniform estimate of the form
\[
    \dtv =\theta J(\theta) \left(1+O\left(\lambda^{-1}
    \right)\right),
\]
as $\lambda\to\infty$, was recently derived in \cite{HZ08}, where
\[
    J(\theta) := \frac1\theta\left(\Phi\left(\sqrt{
    \frac1\theta\log\frac1{1-\theta}}\right)-
    \Phi\left(\sqrt{\frac{1-\theta}\theta\log\frac1{1-\theta}
    }\right)\right),
\]
$\Phi$ being the standard normal distribution function. Other more
general and more uniform approximations were also derived in
\cite{HZ08}.

\subsection{The Wasserstein distance}

Deheuvels and Pfeifer \cite{DP88} proved the asymptotic equivalent
$d_W \sim \lambda_2/\sqrt{2\pi \lambda}$, when
$\lambda_2/\sqrt{\lambda}\to\infty$, improving earlier results in
Deheuvels and Pfeifer \cite{DP86b}. They also obtained many other
estimates, including the following second-order one
\[
    \left| d_W - \lambda_2 e^{-\lambda}
    \frac{\lambda^{\lceil \lambda \rceil}}
    {\lceil \lambda \rceil!}\right| \le \frac{2^{5/2}
    \lambda^{1/2} \theta^{3/2}}{1-\sqrt{2\theta}},
\]
for $|\theta|\le 1/2$. Then Witte \cite{Witte90} gave the bound
\[
    d_W\le -\frac{\sqrt{e\lambda}}{2\sqrt{2\pi}}\,\log
    \left(1-2e^{2p_*}\theta\right),
\]
for $\theta<\tfrac12 e^{-2p_*}$. Xia \cite{Xia97} showed that $d_W
\le \lambda_2/\sqrt{\lambda(1-\theta)}$; see also Barbour and Xia
\cite{BX06} for the estimate $d_W\le
8\lambda_2/(3\sqrt{2e\lambda})$. The strongest results including
more precise higher-order approximations were derived by Roos (1999,
2001), where, in particular,
\[
    d_W \le \left(\frac1{\sqrt{2e}}+\frac{8(2-\theta)}
    {5(1-\sqrt{\theta})^2}\,\sqrt{\theta}
    \right) \frac{\lambda_2}{\sqrt{\lambda}}.
\]

For other results in connection with Wasserstein metrics, see
Deheuvels et al.\ \cite{DKPS88}, Hwang \cite{Hwang99}, \v Cekanavi\v
cius and Kruopis \cite{CK00}.

\subsection{The Kolmogorov distance}

It is known, by definition and Newton's inequality (see Comtet
\cite[p.\ 270]{Comtet74} or Pitman \cite{Pitman97}), that $d_K\le
\dtv\le 2d_K$; see Daley and Vere-Jones \cite{DV08}, Ehm
\cite{Ehm91}, Roos \cite{Roos01}. Thus all upper estimates for
$\dtv$ translate directly to those for $d_K$ and vice versa. Also
many approximation results in probability theory for sums of
independent random variables apply to $S_n$. Both types of results
are not listed and discussed here; see for example Arak and Za\u\i
tsev \cite{AZ88}.

Up to now, we only consider non-uniform bounds for $d_K$. However,
effective uniform bounds can be easily derived based on the Fourier
inversion formula
\begin{align}
    d_K &= \sup_m \left| \frac1{2\pi}\int_{-\pi}^{\pi}
    e^{-imt}\frac{\mathbb{E}e^{itS_n}-e^{\lambda(e^{it}-1)}}
    {1-e^{it}} dt \right| \nonumber \\
    &\le \frac1{2\pi} \int_{-\pi}^\pi \frac{e^{\lambda(\cos t-1)}}
    {|1-e^{it}|} \left|\prod_{1\le j\le n} \left(1+p_j(e^{it}-1)
    \right) e^{-p_j(e^{it}-1)}-1\right| dt. \label{dk-ir}
\end{align}
From (\ref{dk-ir}) and (\ref{Ch-ineq1}), we have
\[
    d_K \le \frac{c_1}{\pi} \lambda_2 \int_0^\pi
    \left| 1-e^{it}\right| e^{-\sigma^2 (1-\cos t)} dt,
\]
which, by the simple inequalities $|1-e^{it}|\le |t|$ and $1-\cos
t\ge 2 t^2/\pi^2$ for $t\in[-\pi,\pi]$, leads to
\[
    d_K \le \frac{c_1}{\pi}\lambda_2 \int_0^\infty
    t e^{-2\sigma^2 t^2/\pi^2} dt
    = \frac{c_1\pi \theta}{4(1-\theta)},
\]
where $c_1\pi/4\approx 0.51$. Although this bound is worse than some
known ones such as $d_K\le 0.36\theta$ in Daley and Vere-Jones
\cite{DV08}, its derivation is very simple and self-contained, the
order being also tight. Furthermore, the leading constant $c_1\pi/4$
can be lowered, say to $0.363c_1<0.24$, by a more careful analysis
but we are not pursuing this further here. Note that it is known
that $d_K \sim \theta/(2\sqrt{2\pi e})$, as $\theta=o(1)$, see
Deheuvels and Pfeifer \cite{DP88}, Hwang \cite{Hwang99}, where
$1/(2\sqrt{2\pi e}) \approx 0.121$.

In a little known paper, Makabe \cite{Makabe62} gives a systematic
study of $d_K$ using standard Fourier analysis, improving earlier
results by Kolmogorov \cite{Kolmogorov56}, Le Cam \cite{LeCam60},
Hodges and Le Cam \cite{HL60}. In particular, he first derived a
second-order estimate from which he deduced that $d_K \le 3.7\theta$
and
\[
    d_K \le \frac{\theta}{2}+O\left(\theta^2+p_*\theta
    \right).
\]
For $p_*<1/5$, he also provided a one-page proof of
\[
    d_K \le \frac{5\theta}{4(1-2p_*-5\theta/2)}
    \le\frac{25\theta}{12-50\theta}.
\]

A Le Cam-type inequality of the form $d_K\le 2\lambda_2/\pi$ was
given in Franken \cite{Franken64}, which was later refined to
$d_K\le \lambda_2/2$ in Serfling \cite{Serfling78}; see also Daley
\cite{Daley80}. Franken \cite{Franken64} also proves the estimate
\[
    d_K \le \frac{c}\pi \left(1-e^{-\lambda(1-\theta)}
    \right) \frac{\theta}{1-\theta},
\]
for an explicitly given $c$, as well as higher-order terms for $d_K$
based on Charlier expansions. His bound together with $d_K\le1$
implies $d_K\le 1.9\theta$, improving previous estimates by Le Cam
and Makabe.

Shorgin \cite{Shorgin77} derived an asymptotic expansion for the
distribution of $S_n$; in particular, as a simple application of his
bounds for $|a_j|$ (see (\ref{F-e-exp})) and $|C_k(\lambda,m)|$,
\[
    d_K\le \left(\frac12+\sqrt{\frac\pi 8}\right)
    \frac{\theta}{1-\sqrt{\theta}},
\]
where $1/2+\sqrt{\pi/8}\approx 1.31$. In Hipp \cite{Hipp85}, the
upper bound
\[
    d_K\le \frac{\pi}{4\lambda(1-\theta)}\,\sum_{1\le j\le n}
    \frac{p_j^2}{1-p_j},
\]
was given, so that if $p_*\le1/4$, then
\[
    d_K\le \frac{\pi \theta}{3(1-\theta)}\le
    \frac{1.05\theta}{1-\theta}.
\]
A bound of the form
\[
    d_K \le \frac{2}{\pi}\min\left\{ \frac{\sqrt{e}\theta}
    {2(1-\theta)}, \lambda_2\right\}
\]
was given in Kruopis \cite{Kruopis86}, where he also derived
\[
    \sup_m\left|\mathbb{P}(S_n\le m)
    -[z^m]\frac{P_2(z)}{1-z}\right|
    \le \frac23\varpi\lambda_3 \min\left\{\frac{1}
    {\sqrt{\pi}\lambda^{3/2}(1-\theta)^{3/2}},1\right\},
\]
where $\varpi$ is defined in (\ref{varpi}). Deheuvels and Pfeifer
deduced several estimates for $d_K$; in particular (see
\cite{DP88,DPR89})
\[
    \sup_m\left|\mathbb{P}(S_n\le m)
    -[z^m]\frac{P_1(z)}{1-z}\right|
    \le \frac53\left(\frac{\theta^2}
    {(1-\sqrt{\theta})}+\frac{\lambda_3}
    {\lambda^{3/2}}\right);
\]
Note that this can also be written as
\[
    \left|d_K - \frac{\theta}{2} \,e^{-\lambda}
    \max\left\{\frac{\lambda^{\ell_+}}{\ell_+!}
    (\ell_+-\lambda),\frac{\lambda^{\ell_-}}{\ell_-!}
    (\lambda-\ell_-)\right\}\right| \le
    \frac53\left(\frac{\theta^2}
    {(1-\sqrt{\theta})}+\frac{\lambda_3}
    {\lambda^{3/2}}\right),
\]
where $\ell_\pm := \lfloor \lambda+1/2\pm \sqrt{\lambda+1/4}
\rfloor$.

Witte \cite{Witte90} then derived the estimate
\[
    d_K \le \frac{\sqrt{e}(1+\sqrt{\pi/2})e^{2p_*}}
    {2\sqrt{2\pi}(1-e^{2p_*}\theta)}\,\theta,
\]
for $\theta<e^{-p_*}$; see also Weba \cite{Weba99}. Roos
\cite{Roos99,Roos01} gives, among several other fine estimates,
\[
    d_K \le \left(\frac1{2e} +\frac{6}
    {5(1-\sqrt{\theta})}\,\sqrt{\theta}\right) \theta.
\]

Non-uniform estimates are derived in Teerapabolarn and Neammanee
\cite{TN06} for general dependent summands, which is of the form in
the case of $S_n$
\[
    \left|\mathbb{P}(S_n\le m) - e^{-\lambda}\sum_{0\le j\le m}
    \frac{\lambda^j}{j!}\right| \le
    \left(1-e^{-\lambda}\right)\theta \min\left\{
    1, \frac{e^{\lambda}}{m+1}\right\},
\]
generally weaker than our bounds in Theorems~\ref{thm-P0} and
\ref{thm-P1}.

\subsection{The point probabilities}

As for $d_K$ above, the point metric can also be readily estimated
by using the integral representation
\begin{align} \label{dp-ir}
    d_P \le \frac1{2\pi} \int_{-\pi}^\pi e^{\lambda(\cos t-1)}
    \left|\prod_{1\le j\le n} \left(1+p_j(e^{it}-1)
    \right) e^{-p_j(e^{it}-1)}-1\right| dt,
\end{align}
and (\ref{Ch-ineq1}), and we obtain for example
\[
    d_P\le \frac{c_1\pi^{5/2} \theta}
    {8\sqrt{2\lambda} (1-\theta)^{3/2}}.
\]
Classical local limit theorems for probabilities of moderate or
large deviations can also be used to give effective bounds for the
point metric $d_P:= \max_m |\mathbb{P}(S_n=m) -e^{-\lambda}
\lambda^m/m!|$; they are not discussed here.

Results for $d_P$ were derived in Franken \cite{Franken64} but are
too complicated to be described here. Kruopis \cite{Kruopis86} gives
the estimate
\[
    d_P \le \min\left\{\frac{\sqrt{e}\theta}
    {\sqrt{\pi\lambda}(1-\theta)^{3/2}},
    \lambda_2  \right\},
\]
as well as
\[
    \sup_m\left|\mathbb{P}(S_n=m)-[z^m]P_2(z)\right|
    \le \frac{8\varpi}{3\pi}\lambda_3\min\left\{ \frac{1}
    {\lambda^2(1-\theta)^2},\frac43\right\}.
\]
Barbour and Jensen \cite{BJ89} derived an asymptotic expansion; see
also \cite{Barbour87}.

Asymptotically, as $\theta\to0$,
\[
    d_P \sim \frac{\theta}{2\sqrt{2\pi \lambda}},
\]
see Roos \cite{Roos95}, where he also derived a second-order
estimate for $d_P$, which was later refined in \cite{Roos99,Roos01}.
In particular,
\[
    d_P \le \left(\frac12\left(\frac3{2e}\right)^{3/2}
    +\frac{6-4\sqrt{\theta}}
    {3(1-\sqrt{\theta})^2}\,\sqrt{\theta}\right)
    \frac{\theta}{\sqrt{\lambda}}.
\]

A non-uniform bound was given in Neammanee
\cite{Neammanee03a,Neammanee03b} of the form
\[
    \left| \mathbb{P}(S_n=m) - e^{-\lambda}
    \frac{\lambda^m}{m!}\right| \le
    \min\left\{ m^{-1},\lambda^{-1}\right\}\lambda_2,
\]
whenever $\lambda\le 1$.

\end{document}